\documentclass[11pt, letterpaper]{article}
\usepackage{fullpage} 
\usepackage[left=35mm, top=20mm, right=35mm, bottom=25mm]{geometry} 
\usepackage{xspace,xcolor,graphicx,tabularx} 
\usepackage{mathtools,amsthm,amssymb,mathrsfs} 
\usepackage{stmaryrd} 
\usepackage{enumitem} 
\setenumerate[0]{label={\normalfont (\roman*)}}

\usepackage[T1]{fontenc}
\usepackage[UKenglish]{babel}
\usepackage[scaled=0.86]{helvet}
\usepackage{mathptmx}
\DeclareMathAlphabet{\mathcal}{OMS}{ntxm}{m}{n}
\usepackage{dsfont} 
\let\mathbb\relax
\let\mathbb\mathds
\usepackage{stmaryrd} 
\makeatletter
\renewcommand{\paragraph}{%
  \@startsection{paragraph}{4}%
  {\z@}{2.25ex \@plus 1ex \@minus .2ex}{-1em}%
  {\normalfont\normalsize\bfseries}%
}
\makeatother
\usepackage{authblk}

\definecolor{linkblue}{HTML}{001487}
\usepackage[colorlinks=true,allcolors=linkblue]{hyperref}
\usepackage{url}
\usepackage{cleveref}

\newtheorem{theorem}{Theorem}[section]
\newtheorem*{theorem*}{Theorem}
\newtheorem{proposition}[theorem]{Proposition}

\newtheorem{lemma}[theorem]{Lemma}
\newtheorem{claim}[theorem]{Claim}

\theoremstyle{remark}
\newtheorem{remark}[theorem]{Remark}
\theoremstyle{definition}
\newtheorem{definition}[theorem]{Definition}

\numberwithin{equation}{section}

\newcommand{\setft}[1]{\textnormal{#1}}
\newcommand{\eps}{\epsilon}

\newcommand{\id}{\setft{id}}
\newcommand{\C}{\ensuremath{\mathds{C}}}
\newcommand{\N}{\ensuremath{\mathds{N}}}

\newcommand{\Z}{\ensuremath{\mathds{Z}}}

\newcommand{\ot}{\ensuremath{\otimes}}

\newcommand{\Tr}{\mathrm{Tr}}
\newcommand{\tr}{\mathrm{tr}}
\newcommand{\rk}{\mathrm{rank}}

\newcommand{\norm}[1]{\left\lVert#1\right\rVert}

\newcommand{\cyl}[1]{\llbracket #1\rrbracket}

\DeclareMathOperator{\Id}{\mathrm{I}}

\DeclareMathOperator{\Tow}{Tow}
\DeclareMathOperator{\per}{per}
\DeclareMathOperator{\aper}{aper}
\DeclareMathOperator{\ext}{ext}
\DeclareMathOperator{\HS}{HS}
\DeclareMathOperator{\Proj}{Proj}

\newcommand{\SRad}{{\texttt{SRad}}}
\newcommand{\SRate}{{\texttt{SRate}}}

\DeclarePairedDelimiterX\braket[2]{\langle}{\rangle}{#1 \delimsize\vert #2}

\newcommand{\cH}{\ensuremath{\mathcal{H}}}

\newcommand{\acts}{\curvearrowright}

\DeclareSymbolFont{greekletters}{OML}{ntxmi}{m}{it}
\DeclareMathSymbol{\alpha}{\mathord}{greekletters}{"0B}
\DeclareMathSymbol{\beta}{\mathord}{greekletters}{"0C}
\DeclareMathSymbol{\gamma}{\mathord}{greekletters}{"0D}
\DeclareMathSymbol{\delta}{\mathord}{greekletters}{"0E}
\DeclareMathSymbol{\epsilon}{\mathord}{greekletters}{"0F}
\DeclareMathSymbol{\zeta}{\mathord}{greekletters}{"10}
\DeclareMathSymbol{\eta}{\mathord}{greekletters}{"11}
\DeclareMathSymbol{\theta}{\mathord}{greekletters}{"12}
\DeclareMathSymbol{\iota}{\mathord}{greekletters}{"13}
\DeclareMathSymbol{\kappa}{\mathord}{greekletters}{"14}
\DeclareMathSymbol{\lambda}{\mathord}{greekletters}{"15}
\DeclareMathSymbol{\mu}{\mathord}{greekletters}{"16}
\DeclareMathSymbol{\nu}{\mathord}{greekletters}{"17}
\DeclareMathSymbol{\xi}{\mathord}{greekletters}{"18}
\DeclareMathSymbol{\pi}{\mathord}{greekletters}{"19}
\DeclareMathSymbol{\rho}{\mathord}{greekletters}{"1A}
\DeclareMathSymbol{\sigma}{\mathord}{greekletters}{"1B}
\DeclareMathSymbol{\tau}{\mathord}{greekletters}{"1C}
\DeclareMathSymbol{\upsilon}{\mathord}{greekletters}{"1D}
\DeclareMathSymbol{\phi}{\mathord}{greekletters}{"1E}
\DeclareMathSymbol{\chi}{\mathord}{greekletters}{"1F}
\DeclareMathSymbol{\psi}{\mathord}{greekletters}{"20}
\DeclareMathSymbol{\omega}{\mathord}{greekletters}{"21}
\DeclareMathSymbol{\varepsilon}{\mathord}{greekletters}{"22}
\DeclareMathSymbol{\vartheta}{\mathord}{greekletters}{"23}
\DeclareMathSymbol{\varpi}{\mathord}{greekletters}{"24}
\DeclareMathSymbol{\varrho}{\mathord}{greekletters}{"25}
\DeclareMathSymbol{\varsigma}{\mathord}{greekletters}{"26}
\DeclareMathSymbol{\varphi}{\mathord}{greekletters}{"27}

\begin{document}

\title{Polynomial Hilbert--Schmidt stability of the lamplighter group}

\author{Alon Dogon and Thomas Vidick}

\date{\vspace{-1cm}}

\maketitle

\begin{abstract}
We establish explicit polynomial bounds on the stability rate and radius of the lamplighter group. 
This provides the first example of an infinitely presented group with an explicit upper bound on the stability radius, answering a question of the first author, Levit and Vigdorovich.
Our approach is based on new dynamical notions, including the analysis of \emph{approximately invariant measures}. 
We establish an effective continuous tower decomposition procedure for approximately invariant measures with presence of periodic points. 
To achieve polynomial bounds we appeal to recent techniques from descriptive combinatorics and distributed LOCAL algorithms.

\end{abstract}


\section{Introduction}

The \emph{normalized Hilbert--Schmidt norm} on matrices is defined by the formula
$$ \|A  \|_{\HS} = \sqrt{\tr(A^*A)} \text{ for $A \in M_d(\mathbb C)$,}$$
where $\tr(A) = \frac{1}{d} \Tr(A)$ is the \emph{normalized trace}. 
In this paper, we give a new example of a group with effective bounds on its Hilbert--Schmidt stability.
In our context, stability refers to the situation where \emph{approximate group representations}, in the Hilbert--Schmidt sense, can be shown to be close to \emph{genuine} group representations.
This can be traced back to a general question of Ulam \cite{Ulam}: ``Is a map between groups which is $\delta$-close to being a homomorphism, necessarily $\eps$-close to a true homomorphism?''. 
A classic instance of the problem is of almost commuting matrices, which arose in work of von Neumann on foundations of quantum mechanics \cite{vNeu}. Polynomial bounds for stability were established in this case \cite{Glebsky}.
See \cite{Dog_Thesis} for an introduction to Hilbert--Schmidt stability.

We will consider \emph{quantitative} and \emph{efficient} forms of Ulam's question: How small does $\delta$ need to be in terms of $\epsilon$? 
We focus on the \emph{lamplighter group}, which can be defined by the following infinite presentation with two generators: 
$$\mathbb Z/2 \wr \mathbb Z = \langle a,t | a^2, \; [a,t^nat^{-n}]\; n \in \N \rangle.$$
Our main result, stated below in a self-contained manner, provides polynomial bounds on the stability of this group: 
    
\begin{theorem}\label{thm:main}
Let $0<\kappa$ and let $M = \left\lceil C\kappa^{-21}\right\rceil, \; \eps = c\,\kappa^7/M^2,$
for universal constants $C,c>0$.
Then for every $d \in \mathbb N$ and $A,T \in U(d)$ unitaries such that
$$ \| A^2 - 1 \|_{\HS} \leq \eps\quad \text{ and }\quad \| [A, T^{-i}AT^{i}] \|_{\HS} \leq \eps\quad \text{ for all}\quad 0\leq i\leq 2M\;,$$
there exist unitaries $\widetilde{A},\widetilde{T} \in U(d)$ such that $\widetilde{A}^2=1$ and $\widetilde{T}^{-i}\widetilde{A} \widetilde{T}^{i}$ commutes with $\widetilde{A}$ for all $i \in \Z$, and further:
\begin{align*}
   \| A - \widetilde{A}\|_{\HS} \leq \kappa\;, \quad  \| T - \widetilde{T}\|_{\HS} \leq \kappa\;.
\end{align*}
\end{theorem}

%

In words: For unitaries $A,T$ to be $\kappa$-close to a true representation of the lamplighter, only polynomially many relations in $A,T$ need to be checked with polynomial precision.
To the best of our knowledge, this yields the first example beyond the realm of finite groups \cite{GH, delaSalle, CVY} and  abelian groups \cite{Glebsky, becker2021abelian}  for which quantitative stability has been established.\footnote{For surface groups, a linear bound on \emph{permutation stability} is known \cite{lazarovich2024surface}, while Hilbert--Schmidt stability is still open. There is also a very different regime of \emph{uniform stability}, for which much more quantitative results are known \cite{DOT}.}
Stability of the lamplighter, with no quantitative bounds, was first proved in \cite{LL1, levit2023characters}. It provided the first example of an infinitely presented stable group.

In recent years, quantitative stability played an important role in quantum complexity theory and the study of error-correcting codes, in particular due to its relation with the study of \emph{self-testing} in the context of non-local games \cite{VidickICM}.
This is evidenced in part by efficient stability results for the multi-qubit Pauli group, such as the Pauli basis test and the low individual degree test~\cite{ji2020quantum, VidickICM, delaSalle, CVY, vidick2022almost}, which play a crucial role in the recent result  \textsc{MIP$^*$=RE} \cite{MIP*=RE}.
Yet, results on quantitative stability remain scarce and challenging to establish, especially in the realm of infinite groups.

A general framework of quantitative stability for finitely generated groups $\Gamma = \langle S \rangle$ was developed in \cite{dogon2024characters}, extending the work of Becker and Moshieff on finitely presented groups \cite{becker2021abelian}.
This is captured by the notions of \emph{stability radius growth} $\SRad_\Gamma$ (Definition \ref{def:SRad}) and \emph{stability rate} $\SRate_\Gamma^{S,r}$ (Definition \ref{def:local_global_defect}), which roughly measure the amount of group relations and precision they need to be tested on to guarantee proximity to a true homomorphism.
The following is our main result formulated in these terms.

\begin{theorem}\label{thm:rate-radius-lamplighter}
    Let $\Gamma = \Z/2 \wr \Z$ be the lamplighter group. Then
$\SRad_\Gamma(r) \preceq r^{21}$. Moreover, for every $0<\kappa$, if
\[
    r_\kappa=\left\lceil C\kappa^{-21}\right\rceil,
\]
and $S=\{a,t\}$ is the standard generating set, then 
\[
    \SRate_{\Gamma}^{S,r_\kappa}(\kappa) \geq c\,\kappa^{70}
\]
for universal constants $C,c>0$. 
\end{theorem}

%
While the quantitative bounds stated in Theorem~\ref{thm:rate-radius-lamplighter} are undoubtedly far from optimal, in Appendix~\ref{sec:linear-lower-bound} we show that the stability radius must be at least linear. 
Further, the lamplighter provides the first infinitely presented group with an explicit upper bound on the stability radius growth, answering \cite[Question 12.10]{dogon2024characters}.
Finally, let us remark that the proof works equally well for lamplighters $A \wr \mathbb Z$, where $A$ is a finite abelian group. 
We choose to work with the classical lamplighter for notational simplicity.

\subsection*{On the proof}

While quantitative results are rare, a variety of groups have been shown to be Hilbert--Schmidt stable, including all virtually nilpotent groups \cite{levit2023characters, ES}, finitely generated metabelian groups of the form $N \rtimes \mathbb Z$ \cite{yaari2025density} and diagonal products \cite{dogon2024characters}.
The reason for this contrast is that for proving \emph{qualitative} Hilbert--Schmidt stability, a useful criterion of Hadwin and Shulman \cite{HS_grp} reduces the problem to studying the \emph{character theory}:
$\Gamma$ is Hilbert--Schmidt stable if and only if every character of $\Gamma$ is the limit of a sequence of finite dimensional traces.
This allows one to employ techniques from representation theory, von Neumann algebras and ergodic theory, making the analysis significantly easier.
Unfortunately, the criterion crucially uses a compactness argument, by taking an ultraproduct of a sequence of approximate representations (see \cite{Dog_Thesis}).
This process loses quantitative information regarding the defect of the asymptotic homomorphism, making it unsuitable for quantitative stability.

We overcome this difficulty by working directly with approximate representations of a semidirect product  group. 
This leads to new notions of \emph{approximately invariant measures} for dynamical systems and techniques to analyze them, which are of independent interest.

\paragraph{Approximately equivariant projection-valued measures.}

In general, for a semidirect product $\Gamma = N \rtimes G$, with $N$ abelian, unitary representations $\pi: \Gamma \to B(\mathcal H)$ correspond via the Fourier transform on $N$ to $G$-equivariant projection-valued measures (abbreviated PVM) $E$ on the dual action $G \curvearrowright \widehat{N}$  (see \cite[Proposition 6.36]{Folland}).
This forms the basis for the Mackey machine and full understanding of representation theory in many cases \cite[Section 6.6]{Folland}.

 Denote by $X = \{0,1\}^\mathbb Z$ the Cantor space and let $\alpha: X \to X$ be a homeomorphism. For $M \in \mathbb N$ and $b \in \{0,1\}^{[-M,M]}$, we let $\cyl{b} \subset X$ denote the cylinder set defined by $b$.
Given a finite-dimensional Hilbert space $\cH$ let $U(\cH)$ and $\Proj(\cH)$ be the unitaries and projections on $\cH$ respectively.

\begin{definition}[cf. Definition \ref{def:approx_equiv_PVM}]\label{def:approx_equiv_intro}
For $M \in \N, \eta>0$ and $T \in U(\mathcal H)$, a projection-valued measure $E$ on $X$ with values in $\Proj(\mathcal H)$ is said to be \emph{\textbf{$(M,\eta)$-equivariant with respect to $T$}} if
    $$\sum_{x \in \{ 0,1\}^{F_M}} \|E_{\alpha\cyl{x}} - T^*E_{\cyl{x}}T \|_{\HS}^2 \leq \eta.$$
Similarly, a probability measure $\mu$ on $X$ is \textbf{\emph{$(M,\eta)$-invariant}} if 
    $\sum_{x \in \{ 0,1\}^{F_M}} |\mu(\alpha\cyl{x}) - \mu(\cyl{x}) | \leq \eta.$
\end{definition}

Inspired by the Mackey machine, we  show that an approximate representation of the lamplighter $\mathbb Z/2 \wr \mathbb Z = \bigoplus_\mathbb Z\mathbb Z/2 \rtimes \mathbb Z$ gives rise to an approximately equivariant PVM $E$ for the \emph{full shift} map $L:  X \to X$, with control on $M$ and $\eta$ (Section \ref{sec:rep_to_meas_main_proof}).
This in turn gives rise to a $(M,\eta)$-invariant measure $\mu(\cyl{x}) = \tr(E_{\cyl{x}})$ on $X$.
To do so, we use quantitative stability of the finite abelian groups $(\mathbb Z/2)^M$ (Lemma \ref{lem:ab-close}), together with the Fourier transform over them.
We expect this correspondence to hold for more general semidirect products.

\paragraph{Approximate projection towers.}

 Shifting our attention to approximately equivariant PVMs E, we first deal with (approximate) \emph{projection towers}.
These arise from restricting $E$ to \emph{dynamical towers} in the system $L:X \to X$:
Under suitable assumptions  (see Section \ref{sec:tow_to_rep_main_proof}), 
if $b \subset X$ is a clopen set with $b, Lb, \dots L^{j-1}b$ pairwise disjoint for some \emph{height} $j$, then the projections $E_b, E_{Lb}, \dots E_{L^{j-1}b}$ are pairwise orthogonal and  satisfy the following.

\begin{definition}\label{def:proj_tower_intro}[cf. Definition \ref{def:proj-tower}]
    A tuple $\tau = (P_0,\ldots,P_{j-1};R)$ is a \emph{\textbf{closed projection tower}}  if $P_0,\ldots,P_{j-1}$ are pairwise orthogonal projections and $R$ is a unitary such that $R^*P_i R = P_{i+1\text{(mod $j$)}}$.
    It is an \textbf{\emph{approximate $(\delta_1,\delta_2)$-closed projection tower of height $j$}} if
    $$\sum_{i=0}^{j-2} \| R^*  P_i R  - P_{i+1}\|_{\HS}^2  \leq \delta_1\;, \qquad
     \| R^*  P_{j-1} R  - P_{0}\|_{\HS}^2 \leq \delta_2\notag\;.$$
\end{definition}

We show that closed projection towers naturally give rise to representations of semidirect products (see Lemma \ref{lem:tower_to_rep} and Remark \ref{rem:tower_to_rep}).
We then directly prove an effective stability lemma for approximate projection towers (Lemma \ref{lem:lowd}):
it is shown that approximate projection towers are close to closed projection towers, depending on how small $\delta_1,\delta_2$ are.
Therefore, it remains to find a suitable decomposition of $X$ into dynamical towers to correct the PVM globally.

\paragraph{Approximately invariant measures and polynomial tower decomposition.}
In order to correct an approximately equivariant PVM on all of $X$, we employ ideas from dynamics and descriptive combinatorics to construct a \emph{polynomial} Kakutani-Rokhlin tower decomposition for approximately invariant measures on the full shift (Proposition \ref{prop:decomp}).
This is the core of our proof.
Below we give an informal statement, and refer to Proposition \ref{prop:decomp} for a precise formulation and Section \ref{sec:Ber} for our notions of complexity of clopen sets, towers, and $\delta$-closedness.

\begin{proposition}[Tower decomposition, informal]\label{prop:into_decomp}
Let $t\in\mathbb{N}$ and let 
$0<\upsilon, \delta, \eta$ be parameters, there exists $M = Poly(t,1/\nu, 1/\delta)$ such that the following holds:
For any probability measure $\mu$ on $X = \{0,1\}^{\mathbb{Z}}$ which is $(M,\eta)$-invariant, there exists a clopen subset $e\subseteq X$ of complexity $M$ with $\mu(e) = Poly(t,\nu,\delta,\eta)$ and a partition of $X\backslash e$ into a finite collection $\mathscr T$ of disjoint clopen towers such that:
\begin{enumerate}
\item  \textbf{(Complexity upper bound)} Each  tower in $\mathscr T$ is of complexity at most $M$
\item \textbf{(Height versus closedness dichotomy)} Each tower appearing in $\mathscr T$ is either $\delta$-closed or of height $t\leq j < 6t+1$.
\item \textbf{(Complexity lower bound)} For each tower $b,Lb, \dots, L^{j-1}b$ in $\mathscr{T}$, the complexity of $b$ is at least $\max(j,t)$.
\end{enumerate}
\end{proposition}

Unlike classical considerations, where the homeomorphism is assumed to be aperiodic, our decomposition accounts for approximately periodic points (see Section \ref{subsec:per}), and has the further advantage of the decomposition consisting only of clopen sets which are of polynomial complexity in terms of the desired height.
The presence of periodic points in our arguments can be traced back to the \emph{dense periodic measures} property of the full shift \cite{levit2023characters}, which is responsible for qualitative Hilbert-Schmidt stability of the lamplighter.

The proof of Proposition \ref{prop:into_decomp} first proceeds by covering the clopen set of \emph{approximately periodic points} (Definition \ref{def:approx_per}) by $\delta$-closed towers of polynomial complexity.
This is done in Section \ref{subsec:per} and relies on the simple yet crucial fact that finite bitstrings can be completed to periodic points in $X$, which can be interpreted as an ``orbit closing lemma'' \cite{Par}.
It is interesting to point out that this step does not use the assumption regarding approximate invariance of the measure $\mu$, but rather purely symbolic dynamical properties of the shift $L:X \to X$.

To deal with the approximately aperiodic part of $X$, we adapt the technique of constructing \emph{marker sets}, a notion appearing in ergodic theory, topological and Borel dynamics and descriptive combinatorics  \cite{GJ, schneider2013locally, GJKS}.
Roughly speaking, these are subsets of $Z\subset X$ which saturate $X$ and have prescribed return time under the dynamics (Section \ref{sec:marker-efficient}), making them useful for construing Kakutani-Rokhlin towers.

To obtain polynomial bounds, we prove an \emph{efficient} set construction (see Proposition \ref{prop:marker_general}) for general group actions on $0$-dimensional spaces, due to P. Naryshkin \cite{Petr_arg}.
The main idea is to make use of Linial's distributed LOCAL color-reduction algorithm \cite{Linial}, together with a recent result of Bernshteyn \cite{Bernshteyn_inv} to yield a highly local and efficient coloring procedure for Schreier graphs of group actions.
Combined with a standard maximal independent set greedy algorithm, we obtain the desired marker set $Z$ with polynomial complexity.

 Finally, we use $Z$ to construct towers of height $\geq t$ via the classical Kakutani-Rokhlin return time procedure, and use approximate invariance of $\mu$ to bound the size of the error set $e$ missed by the towers (Section \ref{subsec:aper}). 

\subsection*{Organization of the paper}
We start with some preliminaries and notation regarding the full shift and representations of the lamplighter group in Section~\ref{sec:prelim}. In Section~\ref{sec:proj-towers} we introduce our notion of approximately equivariant PVM and show a stability result for them (Lemma~\ref{lem:lowd}). In Section~\ref{sec:proof-main} we give the proof of the main result, assuming the main underlying dynamical result (Proposition~\ref{prop:decomp}), which is proved in Section~\ref{sec:decomp}. Section~\ref{sec:marker-efficient} contains the construction of marker sets on which the dynamical results in Section~\ref{sec:decomp} rely. Finally, in Section~\ref{sec:radius} we derive our bounds on the stability rate and radius of the lamplighter, and in Appendix~\ref{sec:linear-lower-bound} we show a linear lower bound on on the stability radius.

\subsection*{Formalization} 
After the paper was complete, the authors used Aristotle~\cite{achim2025aristotle} to auto-formalize in Lean the main statement of the paper, Theorem~\ref{thm:main}. This was completed over 14 (partial) days and 64 Aristotle prompts, each taking between 30 minutes and 4h to complete (the formalization includes all background material not already in Mathlib, such as Lemma~\ref{lem:ab-close}, which is taken from~\cite{CVY}, and Lemma~\ref{lem:loca_alg_col_red}, which is taken from~\cite{Linial}). The complete formalization is available at~\cite{Lean-proof}. A comparator file is available at~\cite{Lean-compare}, which allows easy verification based only on the Lean formalization of the statement of Theorem 1.1.

\subsection*{Acknowledgments}
We would like to thank Itamar Vigdorovich, Arie Levit, Omri Sarig and Tali Kaufman for helpful conversations, and Michael Chapman for comments on the draft.
The first author would like to thank Brandon Seward and Petr Naryshkin who introduced him to Borel dynamics and LOCAL algorithms. He would also like to thank Oren Becker for introducing him to quantitative stability.
AD was supported by a Clore Scholars grant and ERC grant no. 882751. TV was supported by the Swiss
State Secretariat for Education, Research and Innovation (SERI). 

\textbf{A.I. disclosure}
The authors originally achieved an exponential bound on the stability of the lamplighter, with a greedy procedure to construct marker sets for $\mathbb Z$-actions. 
Upon input of the marker lemma to ChatGPT 5.5, a polynomial improvement was achieved.
The authors then observed that the polynomial marker lemma follows from known techniques in descriptive combinatorics and holds for general group actions. The LLM also provided the statement and proof of the lower bound described in Appendix~\ref{sec:linear-lower-bound}. The proof was edited and verified by the authors. 
Besides this input, the paper was written solely by the authors.


\section{Preliminaries}
\label{sec:prelim}

\subsection*{Standing notations}\label{subsec:notation} 
\label{notations}

\begin{itemize}
\item Throughout, $\mathcal{H}$ is a finite-dimensional Hilbert space. 
    \item  We denote the commutator of two operators $A,B$ by $[A,B]= AB-BA$.
    \item  We adopt the standard big O notation, where unless otherwise noted, $x = O(y)$ means $x \leq C\cdot y$ for a universal constant $C$ and $x,y$ are numbers or functions. Similarly we consider the standard $\Omega, \Theta$ notations.
    \item  Let $f,g: (0,\infty) \to (0,\infty)$ be a pair of  monotone non-decreasing functions (similarly for $f,g: \N \to (0,\infty)$).
    \begin{itemize}
        \item We write $f \preceq g$ if there is a constant $C \ge 1$ such that  $f(x) \leq Cg(Cx)$ for all $x>0$. 
        \item We write $f \approx g$  if $f \preceq g$  and $g \preceq f$. In that case  $f$ and $g$ are said to be \emph{equivalent}.
    \end{itemize}
\end{itemize}

\subsection{Approximately invariant measures}
\label{sec:Ber}

We now introduce notations regarding the full shift on $\Z$, which will be fixed throughout the paper.
Let $X = \{0,1\}^\mathbb Z$, and let $L : X \to X$ be the \emph{right shift} defined by:
$$ (Lx)_i = x_{i-1},  \;\; x \in X, i \in \Z.$$
We will refer to $X$ with the $\mathbb Z$ action induced by $L$ as the \emph{\textbf{full shift}} (as is common in topological dynamics), or the \textbf{\emph{(topological) Bernoulli shift}}.

Even though we will only be interested in the full shift, the following definitions make sense for general homeomorphisms of the Cantor space $X$.
For the rest of this subsection, we will fix $\alpha: X \to X$ a general homeomorphism of $X$, which can be taken to be $L$ for simplicity\footnote{In fact, the definitions extend naturally to the setting of a finitely generated group action on compact $0$-dimensional spaces with equipped with continuous colorings. In particular, $X$ can be taken to be $A^\mathbb Z$ for any finite alphabet $A$. We will not need approximately invariant measures in the general setting, though this approach is partially undertaken in Section \ref{sec:marker-efficient}.}.

We equip $X$ with the standard topology induced by \emph{cylinder sets}: For $n \in \mathbb{N}$, let $F_n = [-n,n] \subset \Z$ and consider the projection map $\pi_n: X \to \{0,1\}^{F_n}$ defined by restricting to the coordinates of $x$ supported on $F_n \subset \mathbb Z$.
For $b \in \{0,1\}^{F_n}$, denote by $\cyl{b} = \pi_n^{-1}(b) \subset X$. We will refer to $\cyl{b}$ as a \textbf{\emph{cylinder set}}.
Similarly, if $B \subset \{0,1\}^{F_n}$, we will denote $\cyl{B} = \pi_n^{-1}(B)$. A set of the form $\cyl{B}$ for some $B \subset \{0,1\}^{F_n}$ will be said to be \textbf{\emph{$n$-defined}}.
By Tychonoff's theorem, the space $X$ equipped with the topology generated by cylinder sets is compact, and the collection of clopen subsets coincides with the union of $n$-defined sets for all $n$.
Throughout the paper, $n$-definability measures the complexity of clopen subsets in $X$.

For $t \in \N$, a point $x \in X$ is \textbf{\emph{$t$-periodic}} if $\alpha^t(x) = x$. We will denote by $X_{\per}(t) \subset X$ the set of points which have period at most $t$. Note that $X_{\per}(t)$ is closed, and set $X_{\aper}(t)  = X \setminus X_{\per}(t)$. 

A \textbf{(clopen) \emph{tower} with base $\boldsymbol{b}$ and height $\boldsymbol{j}$}, denoted by $\Tow(b,j)$, is a disjoint union of the form $\Tow(b,j) =  \sqcup _{i=0}^{j-1} \alpha^{i}(b)$, where $b \subset X$ is a clopen subset and $j \in \N$.
We say the tower $\Tow(b,j)$ is \emph{\textbf{$n$-defined}} if all the sets $b, \alpha(b), \dots , \alpha^{j-1}(b)$ are $n$-defined.
Furthermore, we say the tower $\Tow(b,j)$ is \textbf{$\delta$-closed} if $\mu(\alpha^j(b) \cap b) \geq (1-\delta)\mu(b)$.
A subset $b \subset X$ is said to be \emph{\textbf{$F_j$-independent}} if $b \cap \alpha^i(b) = \emptyset$ for all $i \in F_j \setminus \{0\}$. 
Observe that $b \subset X$ is $F_{j-1}$-independent if and only if $b$ is the base of a tower of height $j$.

As mentioned in the introduction, a crucial notion in our techniques is that of \emph{approximately invariant measures} on the full shift, which we now define formally:
\begin{definition}[Approximately invariant measure]\label{def:approx_inv_meas}
Let $\alpha:X\to X$ be a homeomorphism, $M \in \N, \eta>0$. A probability measure $\mu$ on $X = \{0,1\}^\Z$ is said to be \emph{\textbf{$(M,\eta)$-invariant}} if
\begin{equation*} 
    \sum_{x \in \{ 0,1\}^{F_M}} |\mu(\alpha\cyl{x}) - \mu(\cyl{x}) | \leq \eta.
\end{equation*}
In other words, the marginals of $\mu$ and $\alpha_* \mu$ on $\{0,1\}^{F_M}$ are $\eta$-close in total variation.
\end{definition}

Similarly, we can also define \emph{approximately equivariant projection valued measures (PVM)}\footnote{Alternatively, the term \emph{approximate system of imprimitivity} may also be used. It is also natural to consider the same notion for general group actions.} as follows. Given a Hilbert space $\mathcal H$, we denote by $\Proj(\mathcal H)$ the set of projections on $\mathcal H$.
Recall a \emph{projection valued measure $E$ on $X$ with values in $\Proj(\mathcal H)$} is an assignment of a projection $E_b \in \Proj(\mathcal H)$ for each Borel subset $b \subset X$ which is countably additive.

\begin{definition}[Approximately equivariant PVM]\label{def:approx_equiv_PVM}
Let $\alpha: X \to X$ be a homeomorphism, $M \in \N, \eta>0$, $\mathcal H$ a Hilbert space and $T \in U(\mathcal H)$. A projection valued measure $E$ on $X$ with values in $\Proj(\mathcal H)$ is said to be \emph{\textbf{$(M,\eta)$-equivariant with respect to $T$}} if
\begin{equation*} 
    \sum_{x \in \{ 0,1\}^{F_M}} \|E_{\alpha\cyl{x}} - T^*E_{\cyl{x}}T \|_{\HS}^2 \leq \eta.
\end{equation*}
\end{definition}

Naturally, approximately equivariant PVMs induce approximately invariant measures:

\begin{lemma}[From PVM to measures]\label{lem:PVM-to-meas}
Let $M \in \N, \eta >0$. Given a PVM $E$ on $X$ with values in $\Proj(\mathcal H)$ which is $(M,\eta)$-equivariant with respect to $T \in U(\mathcal H)$, the probability measure defined by
$$ \mu(b) = \tr(E_b) \quad\text{ $b\subset X$ Borel}\;,$$
is $(M,\eta)$-invariant.
\end{lemma}
\begin{proof}
    Using Lemma \ref{lem:proj_diff_bound} below and the fact that $\mu(\cyl{x}) = \tr(E_{\cyl{x}}) = \tr(T^*E_{\cyl{x}}T)$ and $\mu(\alpha\cyl{x}) = \tr(E_{\alpha\cyl{x}})$ for all $x \in \{ 0,1\}^{F_M}$, we have:
    \begin{align*}
        \sum_{x \in \{0,1\}^{F_{M}}} &|\mu(\cyl{x}) - \mu(\alpha\cyl{x})| =\sum_{x \in \{0,1\}^{F_{M}}} | \tr(T^*  E_{\cyl{x}} T) - \tr(E_{\alpha\cyl{x}}) | \\
    &\leq\sum_{x \in \{ 0,1\}^{F_M}} \|E_{\alpha\cyl{x}} - T^*E_{\cyl{x}}T \|_{\HS}^2 \leq \eta\;.
    \end{align*}
\end{proof}

The following bound was used in the proof of Lemma~\ref{lem:PVM-to-meas}, and is generally useful to compare traces of different projections in terms of the square Hilbert--Schmidt distance between them. 
It improves the Cauchy--Schwarz bound $| \tr(P) - \tr(Q) | \leq \| P-Q\|_{\HS}$ with an extra factor. 

\begin{lemma}\label{lem:proj_diff_bound}
    For any two projections $P,Q$, we have
    $$|\tr(P) - \tr(Q) | \leq \| P-Q\|_{\HS}^2$$
\end{lemma}
\begin{proof}
Note that $\tr(P(1-Q)P) = \tr(P) - \tr(PQ)$ and $\tr(Q(1-P)Q) = \tr(Q) - \tr(PQ)$ are both positive. Consequently: 
\begin{align*}
    | \tr(P) - \tr(Q) | &= |\tr(P(1-Q)P) - \tr(Q(1-P)Q)| \\
    &\leq \tr(P(1-Q)P) + \tr(Q(1-P)Q) = \| P-Q\|_{\HS}^2.
\end{align*}
\end{proof}

\subsection{Representation theory of the lamplighter group}\label{sec:rep-theory}

We now turn to describing the finite dimensional representations of the lamplighter group, our main object of study.

Let $\Gamma = \Z/2 \wr \Z = \langle a,t | a^2, \; [a,t^nat^{-n}]\;n \in \N \rangle$ be the lamplighter.
Denote by $G = \langle t \rangle \simeq \Z$ the cyclic subgroup of infinite order, and let $N = \bigoplus_{\Z} \Z/2$ be the lamp group, so that $G \acts N$ by the left shift, and $\Gamma = N \rtimes G$.
Let $\widehat{N}$ be the Pontryagin dual, so that $\widehat{N} \simeq X = \prod_\Z \Z /2$.
The action $G \acts N$ induces a dual action $G \acts \widehat{N}$, which is isomorphic to the full right shift $L:X \to X$ (see Section \ref{sec:Ber}).

\begin{definition}\label{def:rep-of-lamp}
Fix $d \in \N$, and let $\lambda\in \mathbb{T}$ and $y \in \{0,1\}^{[0,d-1]}$. Consider the matrix
$$C_{d,\lambda} = \begin{pmatrix} 0 & 0 & 0 & \cdots & 0 & \lambda \\ 1 & 0 & 0 & \cdots & 0 & 0 \\ 0 & 1 & 0 & \cdots & 0 & 0 \\ \vdots & \vdots &  \ddots &\ddots &\vdots & \vdots\\ 0& 0 & 0& \cdots & 1 & 0 \end{pmatrix}$$
and define the unitary representation $\pi_{y,\lambda}: \Gamma \to U(d)$ by
\begin{align*}
    \pi_{y,\lambda}(a) &= \mathrm{diag}((-1)^{y_0},  \dots, (-1)^{y_{d-1}}),\\
    \pi_{y,\lambda}(t) &= C_{d,\lambda}.
\end{align*}
\end{definition}

Note that $C_{d,\lambda}^d = \lambda\Id$, so that $\pi_{y,\lambda}$ has finite image if and only if $\lambda$ is a root of unity.
Further, $\pi_{y,\lambda}$ is irreducible if and only if $y$ is the projection of an element $x \in X_{\per}(d) \setminus X_{\per}(d-1)$ to the $[0,d-1]$-coordinates.
The collection $\{\pi_{y,\lambda}\}$ ranging over all $d \in \mathbb N, \lambda \in \mathbb T$ and $y \in \{0,1\}^{[0,d-1]}$ as above forms all isomorphism classes of finite-dimensional irreducible representations of the lamplighter group (see \cite[Corollary 4.B.9]{BdlH}).


 We mention an alternative way, via the Mackey machine, to describe finite dimensional irreducible representations using periodic elements in the full shift (as in \cite[\S 3.E]{BdlH}): 
 Fix $d \in \N$, and let $G(d) = \langle t^d \rangle$.
 Fix $\lambda \in \mathbb T$, and  $x \in X_{\per}(d) \setminus X_{\per}(d-1)$.
 Let $\chi_\lambda \in \widehat{G(d)}$ be the character defined by $\chi_\lambda(t^{dn}) = \lambda^n$, and let $\chi_x \in \widehat{N}$ be the character corresponding to $x$.
 Then we can define the character $\chi_{\lambda, x} (t^n, \xi) = \chi_\lambda(t^n)\chi_x(\xi)$.
 The induced representation $\mathrm{Ind}_{N \rtimes G(d)}^\Gamma \chi_{\lambda, x}$ is exactly the representation $\pi_{y,\lambda}$, where $y$ is the projection of $x$ to the coordinates $[0,d-1]$.
 The advantage of this approach is that it extends easily to any metabelian group of the form $N \rtimes \mathbb Z$ as in \cite{yaari2025density}.

%

\section{Projection towers}
\label{sec:proj-towers}

Given the structure of representations of the lamplighter group, and more generally, semidirect products as evidenced in Section~\ref{sec:rep-theory}, the following concept naturally plays an important role in our arguments. 

\begin{definition}[Projection towers]\label{def:proj-tower}
    Let $\mathcal H$ be a Hilbert space. We call a tuple $\tau = (P_0,\ldots,P_{j-1};R)$ a \textbf{\emph{projection tower of height $j$ on $\mathcal H$}} if $P_0,\ldots,P_{j-1}$ are pairwise orthogonal projections and $R$ a unitary on $\mathcal H$ such that
    $$R^*  P_i R = P_{i+1} \text{ for all $0\leq i< j-1$}\;.$$ 
    In this case, we will refer to $P_0$ as the \emph{\textbf{base}} of $\tau$, and to the projection $P_\tau = \sum_{i=0}^{j-1}P_i$ as the \emph{\textbf{support}} of $\tau$.
    Further, we say $\tau$ is \textbf{\emph{closed}} if $R^*P_{j-1}R = P_0$.
\end{definition}

\begin{remark}
    We have introduced two notions of towers, the projection towers of Definition~\ref{def:proj-tower} and the clopen towers appearing in the full shift (Section~\ref{sec:Ber}). The relation between them, via projection valued measures, will be explained later in Section \ref{sec:proof-main}. We will often refer to either object as a ``tower'' as it will generally be clear from context. 
\end{remark}

Note that if $\tau$ is a closed tower, then $P_\tau \mathcal H$ is an invariant subspace of $R$, that is $RP_\tau=P_\tau R$.
The relevance of closed projection towers to the lamplighter group stems from the following simple observation:

\begin{lemma}[From towers to representations]\label{lem:tower_to_rep}
Let $\tau = (P_0,\ldots,P_{j-1};R)$ be a closed projection tower on $\mathcal H$, and let $x \in \{0,1\}^{[0,j-1]} \simeq X_{\per}(j)$. The  operators
\begin{align*}
    A _{\tau,x}&= \sum_{i=0}^{j-1}  (-1)^{x_i} P_i +(I - P_\tau), \\
    T_\tau &=  R^*,
\end{align*}
define a unitary representation $\rho_{\tau,x}$ of the lamplighter $\Gamma$ on $\mathcal H$, by taking $\rho_{\tau,x}(a)=A_{\tau,x},\; \rho_{\tau,x}(t)=T_\tau$.
Moreover, this representation is a direct sum of representations of the form $\pi_{x, \lambda_i}$ (for  some $\lambda_i \in \mathbb T$) as in Definition \ref{def:rep-of-lamp}, together with the identity representation on $(I - P_\tau)\mathcal H$.
\end{lemma}

\begin{proof}
    The fact that $(A_{\tau,x},T_\tau)$ are unitaries which satisfy the defining relations of the lamplighter group is immediate from the definition of a closed projection tower, so that $\rho_{\tau,x}:\Gamma \to U(\mathcal H)$ is a unitary representation.
    Notice that $R^{-j}P_0R^{j} = P_0$, so that $P_0 T^jP_0 = P_0 R^{-j} P_0$ is a unitary on the image of $P_0$. 
    Also note that $P_\tau \mathcal H$ is an invariant subspace of $\rho_{\tau,x}$, and $\rho_{\tau,x}$ acts as the identity on the complement.
    By considering the spectrum of $P_0T^jP_0$, it readily follows that we have a direct sum decomposition of the form
    $$ \rho_{\tau,x} P_\tau = \bigoplus_{\lambda \in \textrm{Spec}(P_0T^jP_0)} \pi_{x,\lambda}\;.$$
\end{proof}

\begin{remark}\label{rem:tower_to_rep}
    It is is easy to extend Lemma \ref{lem:tower_to_rep} to semidirect products of the form $N \rtimes G$, where $N$ is abelian and $G$ is a finitely generated residually finite group, where a closed projection tower is one that arrises from a finite quotient of $G$.
\end{remark}
In the course of the proof we will obtain tuples that are \emph{approximately} a tower, and \emph{approximately closed}, in the following sense. 

\begin{definition}[Approximately closed projection towers]\label{def:approx-proj-tower}
    A tuple $\tau = (P_0,\ldots,P_{j-1};R)$ is an \textbf{\emph{approximate $(\delta_1,\delta_2)$-closed projection tower of height $j$}} if $P_0,\ldots,P_{j-1}$ are pairwise orthogonal projections and $R$ a unitary such that
\begin{align}\label{eq:lowd-0a}
    \sum_{i=0}^{j-2} \| R^*  P_i R  - P_{i+1}\|_{\HS}^2 &\leq \delta_1\;, \\
     \| R^*  P_{j-1} R  - P_{0}\|_{\HS}^2 &\leq \delta_2\notag\;.
\end{align}
In this case, we will refer to $P_0$ as the \emph{\textbf{base}} of $\tau$, and to the projection $P_\tau = \sum_{i=0}^{j-1}P_i$ as the \emph{\textbf{support}} of $\tau$.
\end{definition}

The following introduces a notion of closeness on (approximate) towers. 
The closeness is measured separately for the projections involved, and the tower transformation.

\begin{definition}
   Given $\tau = (P_0,\ldots,P_{j-1};R)$ and $\tau'=(P'_0,\ldots,P'_{j-1};R')$ two (approximate) projection towers of the same height, we say that they are $(\eps_1,\eps_2)$-close if
   \begin{equation}\label{eq:lowd-0b}
    \sum_{i=0}^{j-1} \|  P_i - P'_{i}\|_{\HS}^2 \leq \eps_1\;,
\end{equation}
and 
  \begin{equation}\label{eq:lowd-0c}
      \| P_\tau RP_\tau - P_{\tau'}R'P_{\tau'}\|_{\HS}^2 \leq \eps_2\;.
\end{equation}
\end{definition}

\subsection{Stability for approximate projection towers} \label{subsec:stab_approx_tower}

The following stability lemma is of crucial importance; it will allow us to round an approximately closed approximate projection tower to a true closed projection tower.

\begin{lemma}[Rounding approximate closed projection towers]\label{lem:lowd}
Let $(P_0,\ldots,P_{j-1};R)$ be an approximate $(\delta_1,\delta_2)$-closed projection tower, and let $P$ be its support.

Then there exists a closed projection tower $(P'_0,\ldots,P'_{j-1};R')$ that is $(\eps_1,\eps_2)$-close to it, where 
\[\eps_1 = O\big(j\delta_1 \big)\qquad\text{and}\qquad \eps_2=O\big(j^2\delta_1+\delta_2\big)\;.\]
Furthermore, we can take $P'_i\leq P_i$ for all $0\leq i <j$ and $R'$ is a unitary that acts as the identity on the complement of $P'$. 
\end{lemma}

Let us point out an observation regarding approximate towers that will be relevant to our general proof strategy for stability of the lamplighter --
If $(P_0,\dots, P_{j-1}; R)$ is an approximate tower such that 
$$\sum_{i=0}^{j-2} \| R^*  P_i R  - P_{i+1}\|_{\HS}^2 = \delta_1$$
is sufficiently small, it follows that $\tr(P_0)$ and $\tr(P_{j-1})$ are sufficiently close to $\tr(P)/j$, and so $\| R^*P_{j-1}R - P_0\|_{\HS}^2$ is suitably bounded by $O(\tr(P)/j)$.
This implies that if $\delta_1$ is small enough and $j$ is large enough, then the tower is automatically almost closed.
It thus makes sense to distinguish between "high towers" and "short towers" when applying Lemma \ref{lem:lowd}.

The proof of Lemma \ref{lem:lowd} proceeds in two steps. In the first step, we remove a natural obstruction to $(P_0,\ldots,P_{j-1};R)$ being an exact tower, which is that the rank of the $P_i$ need not be the same, whereas for an exact tower it must be. That is, we identify projections $P'_0,\ldots,P'_{j-1}$ that are close to $P_0,\ldots,P_{j-1}$ and all $P'_i$ have the same rank. In the second step, we modify $R$ to a unitary $R'$ that exactly cyclically conjugates the $P'_i$. 

For the first step, we will rely on the following simple claims. 

\begin{claim}\label{claim:p-bound}
Let $P_0,\ldots,P_{j-1}$ be pairwise orthogonal projections and $P=P_0+\cdots+P_{j-1}$. Then there exists orthogonal projections  $P'_0,\ldots,P'_{j-1}$ each of the same dimension, such that $P'_i \leq P_i$, and 
\begin{equation}\label{eq:p-bound-0}
\sum_{i=0}^{j-1} \| P'_i-P_i\|^2_{\HS}
\leq j\cdot \sum_{i=0}^{j-2} |\tr(P_{i+1})-\tr(P_{i})|\;.
 \end{equation}
 Furthermore, 
 $$\sum_{i=0}^{j-1} | \tr(P_i) - \tr(P) / j| \leq  2j\cdot \sum_{i=0}^{j-2} |\tr(P_{i+1})-\tr(P_{i})|$$
\end{claim}

\begin{proof}
Let $i_0$ be such that $\tr(P_{i_0})$ is minimal out of $P_0, \dots, P_{j-1}$.
Note that for all $i_0<i<j$ we have by the triangle inequality:
$$ |\tr(P_i) - \tr(P_{i_0})| \leq \sum_{k=i_0}^{i-1} |\tr(P_{k+1}) - \tr(P_k)|\leq \sum_{k=0}^{j-2}|\tr(P_{k+1}) - \tr(P_k)|\;.$$
Similarly, the same is true for each $0 \leq i< i_0$.
For each $0\leq i<j$ let $P'_i$ be an arbitrary $\rk(P_{i_0})$-dimensional subspace of $P_i$.
Then by construction $P'_i\leq P_i$ for all $i$ and so
\begin{align*}
\sum_{i=0}^{j-1} \| P'_i-P_i\|^2_{\HS} &= \sum_{i=0}^{j-1} |\tr(P_i)-\tr(P'_i)| = \sum_{i=0}^{j-1} |\tr(P_i)-\tr(P_{i_0})|\\
&\leq j\cdot  \sum_{i=0}^{j-2} |\tr(P_{i+1})-\tr(P_{i})|\;.
\end{align*}
Since $\tr(P_{i_0})$ was minimal, we have $\tr(P_i') \leq \tr(P)/j$ for all $i$. It follows that 
$$\sum_{i=0}^{j-1} (\tr(P)/j- \tr(P_i')) =  \sum_{i=0}^{j-1} \tr(P_i)-\tr(P'_i) \leq j\cdot  \sum_{i=0}^{j-2} |\tr(P_{i+1})-\tr(P_{i})|\;,$$
 so we get the furthermore statement by the triangle inequality.
\end{proof}

\begin{claim}\label{claim:svd}
    Let $R$ be a contraction and $P,Q$ projections of the same dimension such that $RR^* \leq P$, $R^*  R\leq Q$. Then there exists a partial isometry $V$ with  $V^*V = Q, VV^* = P$ such that $\|R-V\|_{\HS}\leq \|R^*  R-Q\|_{\HS}$. 
\end{claim}

\begin{proof} 
Note that the assumptions imply $R=PRQ$.
Since $\rk(P) = \rk(Q)$, applying the polar decomposition ${R}=VT$, with $T$ positive and supported on $Q$, gives a partial isometry $V$ with source projection $Q$ and range $P$. We get 
\[ \|{R}-V\|_{\HS} = \|T-Q\|_{\HS}  = \|({R}^*  {R})^{1/2}-Q\|_{\HS}\leq   \|R^*  R - Q\|_{\HS}\;\]
where the last inequality uses $1-z\leq 1-z^2$ for $z\in[0,1]$. 
\end{proof}

We also recall the following elementary lemma, which will be repeatedly used.

\begin{lemma}[semi-triangle inequality]\label{lem:semitriang}
For any normed space $(V, \| \cdot \|) $ and any $A_1, \dots, A_k \in V$, we have
$$\|\sum_{i=1}^k A_i\|^2 \leq k\sum_{i=1}^k \|A_i\|^2.$$
\end{lemma}

\begin{proof}[Proof of Lemma \ref{lem:lowd}]
To apply the claim we evaluate the right-hand side in~\eqref{eq:p-bound-0}:
   \begin{align*}
   \sum_{i=0}^{j-2} \big|\tr(P_i)-\tr(P_{i+1})\big| &\leq \sum_{i=0}^{j-2} \big|\tr(R^*  P_i R)-\tr(P_{i+1})\big|\\
   &\leq \sum_{i=0}^{j-2} \| R^*P_iR - P_{i+1}\|^2_{\HS} \leq \delta_1
   \end{align*}
where we used Lemma \ref{lem:proj_diff_bound} in the second line and condition~\eqref{eq:lowd-0a} in the definition of an approximate tower.
Applying Claim ~\ref{claim:p-bound} we obtain $P'_i \leq P_i$ each of the same dimension satisfying condition~\eqref{eq:lowd-0b} with $\eps_1=O(j\delta_1)$. 
Further, since $P_i' \leq P_i$, it follows from  ~\eqref{eq:lowd-0b} that $\| P - P'\|_{\HS}^2 = \eps_1$.

In the second step, we define $R'$ and verify~\eqref{eq:lowd-0c}. First, note that 
\begin{align} 
\sum_{i=0}^{j-2}\big\| &(P'_{i} R P'_{i+1})^*(P'_{i} R P'_{i+1})  - P'_{i+1}\big\|_{\HS}^2\notag\\
&\leq O\left(\sum_{i=0}^{j-2}\big\| (P_{i} R P_{i+1})^*(P_{i} R P_{i+1})  - P_{i+1}\big\|^2_{\HS} + \big\| P_{i} - P'_i\|_{\HS}^2 + \big\| P_{i+1} - P'_{i+1}\|_{\HS}^2\right)  \notag\\
&\leq O\left(\sum_{i=0}^{j-2}\big\|  R^* P_{i}R  - P_{i+1}\big\|_{\HS}^2 + \big\| P_{i} - P'_i\|_{\HS}^2 + \big\| P_{i+1} - P'_{i+1}\|_{\HS}^2\right)\label{eq:lowd-2}\\
&= O(\delta_1+\eps_1)\;,\notag
\end{align}
where the second line uses the semitriangle inequality (Lemma \ref{lem:semitriang}), the third and last lines use~\eqref{eq:lowd-0a} together with the fact that $\|P A P \|_{\HS}\leq \|A\|_{\HS}$ for any projection $P$ to bound the first term, and~\eqref{eq:lowd-0b} to bound the others. 
Applying Claim~\ref{claim:svd} individually to each $P_{i}'RP_{i+1}'$, we obtain partial isometries $R'_i$ from $P'_{i+1}$ onto $P'_{i}$ such that 
\begin{equation}\label{eq:lowd-4a}
 \sum_{i=0}^{j-2}\big\| P'_{i} R P'_{i+1} -  R'_i\big\|_{\HS}^2 =  O(\delta_1   +\eps_1)\;.
  \end{equation}
For $i=j-1$, a similar calculation but stopping at~\eqref{eq:lowd-2} gives a partial isometry $R'_{j-1}$ from $P'_{0}$ onto $P'_{j-1}$ such that 
\begin{equation}\label{eq:lowd-4b} \big\|P'_{j-1} R P'_{0} -   R'_{j-1} \big\|_{\HS}^2  = O\big( \big\| R^* P_{j-1}R  - P_0\big\|_{\HS}^2 + \eps_1\big) = O(\delta_2 + \eps_1)\;.
  \end{equation}
Overall, setting $P' = \sum_{i=0}^{j-1}P_i'$, we have that $R'= \sum_{i=0}^{j-1} R'_i + (I - P')$ is a unitary  such that $(P_0',\dots,P_{j-1}';R')$ forms a closed projection tower, and 
\begin{align*}
   \| P  RP-P'&R'P'\|_{\HS}^2 = \|P (RP -R'P')\|_{\HS}^2\\
   &=\Big\| \sum_{i=0}^{j-1} P (RP_i-R'P'_i)\Big\|_{\HS}^2\\
   &\leq 2(j-1)\sum_{i=1}^{j-1} \|P  (RP_i-R'P'_i)\|_{\HS}^2 + \|P  (RP_{0}-R'P'_{0})\|_{\HS}^2\\
  &\leq O\left(j\sum_i\|    P(P_{i-1}RP_i-R'_{i-1})\|_{\HS}^2  +  \|P  (P_{j-1}RP_{0}-R'_{j-1})\|_{\HS}^2+ j\delta_1 + \delta_2 \right)\\
  &\leq O\left(j\sum_i\|    P_{i-1}'RP_i'-R'_{i-1}\|_{\HS}^2  +  \|P_{j-1}'RP_{0}'-R'_{j-1}\|_{\HS}^2+ j\delta_1 + \delta_2 + \eps_1\right)\\
  &= O\left(j(\delta_1 + \eps_1)+   \delta_2 \right)\;,
\end{align*}
where we used $P' \leq P, P_i\leq P$ and the semitriangle inequality (Lemma \ref{lem:semitriang}) repeatedly, as well as the equalities $R'P'_i = R'_{i-1 \text{(mod $j$)}}$.
Further, the inequality on the fourth line uses the assumptions ~\eqref{eq:lowd-0a}, and the last two lines use ~\eqref{eq:lowd-0b} and~\eqref{eq:lowd-4a},~\eqref{eq:lowd-4b}. This shows~\eqref{eq:lowd-0c} with $\eps_2=O(j^2\delta_1+\delta_2)$. 
\end{proof}

\section{Proof of the main results}
\label{sec:proof-main}

In this section we give the proof of Theorem \ref{thm:main}, treating the dynamical machinery developed later in Section \ref{sec:decomp} as a black box.
We will freely use the language introduced in Section \ref{sec:Ber} regarding the full shift space.
Throughout the section, we will fix a target error $\kappa$, and a pair $(A,T)$ of unitaries on $\mathcal H = \mathbb C^d$ (for some $d \in \N$) with

\begin{equation}\label{eq:stand_ass1}
\|A^2-1\|_{\HS} \leq \eps\;, \quad \| [A, T^{-i}AT^{i}] \|_{\HS} \leq \eps \quad\text{for all $0\leq i\leq 2M$}\;,
\end{equation}
where $M, \eps$ are as specified in Theorem \ref{thm:main}.

We start by modifying $A,T$ using standard stability lemmas to arrive at the correct starting point to apply dynamical methods.
Without loss of generality, we may use the slightly stronger assumption:
\begin{equation}\label{eq:stand_ass2}
A^2 = 1\;,\quad\| [A_i, A_j] \|_{\HS} \leq \eps\quad \text{for all $i,j\in F_M$}\;,
\end{equation}
where $A_i = T^{-i}AT^{i}$ for $i \in \Z$, and $F_M = \{-M, \dots, M\}$ as in Section \ref{sec:Ber}.
Indeed, by \cite[Proposition 1.4]{DGLT}, under the weaker assumption~\eqref{eq:stand_ass1} there exists $A' \in U(\mathcal H)$ with $A'^2=1$ and $\|A'-A\|_{\HS} \leq \eps$. Replacing $A$ by $A'$ changes the commutator bounds by at most $O(\eps)$, which we absorb into $\eps$ by adjusting the universal constant $c$.
Further, by unitary invariance of the Hilbert--Schmidt norm, the commutator bound in~\eqref{eq:stand_ass1} for shifts up to $2M$ implies the second condition in~\eqref{eq:stand_ass2}: indeed, for $i,j\in F_M$, the difference $j-i$ has absolute value at most $2M$.

A final preprocessing step will be to replace $A_i$ with nearby, truly commuting unitaries, which will be approximately equivariant with respect to $T$.
To this end, we apply the following stability lemma for the group $\mathbb F_2^n$.

\begin{lemma}[Theorem 3.2 in \cite{chao_et_al:LIPIcs.ITCS.2017.4}]\label{lem:ab-close}
Let $P_1, \dots P_n$ be projections on $\mathcal H$ such that 
$$ \| [P_i,P_j] \|_{\HS} \leq \eps_0\quad \text{for all $i,j \leq n$}\;.$$
Then there exists a family of pairwise commuting projections $Q_1, \dots,  Q_n$ such that
$\|Q_i-P_i\|_{\HS} \leq 8n\eps_0$ for all $i$. 
\end{lemma}

This theorem is derived for the operator norm in \cite{chao_et_al:LIPIcs.ITCS.2017.4}, and subsequently for the Hilbert-Schmidt norm in~\cite[Theorem A.1]{CVY} but the proof works equally well for the Hilbert--Schmidt norm.\footnote{Alternatively, a weaker version with a bound of $O(n^2\eps_0)$ can also be deduced from applying uniform Hilbert--Schmidt stability for the family of abelian groups \cite{DA} (see the discussion after \cite[Lemma 2.11]{CVY}).}
 In our scenario, we apply Lemma \ref{lem:ab-close} to the $n=2M+1$ projections $(A_i + 1)/2$, $i \in F_M$, with $\eps_0=O(\eps)$. 
 We obtain a family of pairwise commuting unitaries $B_i, i \in F_M$, which square to the identity and $\| A_i - B_i \|_{\HS}^2 \leq \eps'$, where $\eps' = O((M\eps)^2)$.
 We have thus arrived at our final standing assumptions:
\begin{align}\label{eq:stand_ass3}
&\| A_i- B_i \|_{\HS}^2 \leq \eps'\;,\quad B_i^2 = 1\;,\quad [B_i, B_j]=0\quad \text{for all $i,j\in F_M$ and } \\
& \| T^*B_iT - B_{i+1}\|_{\HS}^2 \leq \eps'\quad \text{for all $i \in F_{M-1}$}\;, \notag
\end{align}
where the last condition follows from the fact that $T^*A_iT = A_{i+1}$ for all $i \in \Z$ and the semi-triangle inequality (Lemma \ref{lem:semitriang}).
Lemma~\ref{lem:mu-invariance} below converts this to the bound $\eps''=O(M^2\eps')=O(M^4\eps^2)$, which is at most $c\kappa^{14}$ for the choice $\eps=c\kappa^7/M^2$ after decreasing $c$ once more.

\subsection{From approximate representations to approximately invariant measures}
\label{sec:rep_to_meas_main_proof}
We will shift our focus to the projection valued measure on $X=\{0,1\}^{\Z}$ associated with the commuting family $B_i, i \in F_M$. This passage can be interpreted as the Fourier transform for $\mathbb F_2^{F_M}$, although we will not need this interpretation.

Consider the projection valued measure $\bar E:2^{\{0,1\}^{F_M}}\to \Proj(\mathcal H)$, whose values on atoms are defined by
$$\bar E_x = \prod_{i \in F_M} \frac{1}{2}\left(I+ (-1)^{x_i}B_i\right) \text{ for $x \in \{0,1\}^{F_M}$}\;.$$
That is, $\bar E_x$ is the projection onto the joint eigenspace of $B_i, i\in F_M$ corresponding to $(-1)^{x_i},  i \in F_M$.
Recall that we consider $\pi_M: X\to\{0,1\}^{F_M}$ as a quotient via the projection map.
Note that the shift-invariant closed subset $X_{\per}(2M+1) \subset X$ is a section for this map.
Thus, we may pushforward $\bar E$ to a projection valued measure $E$ on $X_{\per}(2M+1)$, and extend it trivially to $X$. 
Observe that we have
$$E_{\cyl{x}} = \bar E_x \text{ for all $x \in \{0,1\}^{F_M}$}\;,$$
since every $x \in \{0,1\}^{F_M}$ has a unique $2M+1$-periodic  extension to $X$. Thus,  we also have:
$$ E_b =  \bar E_{\pi_M(b)} \text{ for all $M$-definable $b \subset X$}\;.$$
Finally, consider the probability measure $\mu$ on the full shift space $X$ defined by $\mu(b) = \tr(E_b)$ for all Borel $b \subset X$ (as in Lemma \ref{lem:PVM-to-meas}).
Our next task is to show that $E$ is an \emph{approximately equivariant PVM} (Definition \ref{def:approx_equiv_PVM}), and $\mu$ is an \emph{approximately invariant measure} (Definition \ref{def:approx_inv_meas}) under the full shift.

\begin{lemma}[Approximate equivariance of $E$ and invariance of $\mu$]\label{lem:mu-invariance}
The projection valued measure $E$ is $(M-1,\eps'')$-equivariant with respect to $T$, that is:
\begin{align}
   \sum_{x\in\{0,1\}^{F_{M-1}}} \big\| T^*  E_{\cyl{x}} T - E_{L\cyl{x}}\big\|_{\HS}^2 &\leq \eps'' \;, \label{eq:clb-1}
\end{align}
where $\eps'' = C(M^2\eps') = C(M^4 \eps^2)$ for a universal constant $C$.
Furthermore, the  measure $\mu$  is $(M-1, \eps'')$-invariant.
\end{lemma}
\begin{proof}
    Start by observing that for $x \in \{0,1\}^{F_{M-1}}$, we have $L\cyl{x} = \cyl{b}$, where $b \subset \{0,1\}^{F_M}$ is the subset of  sequences whose $[-M,-M+1]$ coordinates are arbitrary, and the $[-M+2,  M]$ coordinates are the right shift of $x$. 
    It follows that:
    \begin{align*}
    E_{\cyl{x}} &= \bar E_{\pi_M(\cyl{x})} = \prod_{i \in F_{M-1}} \frac{1}{2}\left(I+ (-1)^{x_i}B_i\right)\;, \\ 
    E_{L\cyl{x}} &= \bar E_{b}= \prod_{i \in F_{M-1}} \frac{1}{2}\left(I+ (-1)^{x_{i}}B_{i+1}\right)\;.
    \end{align*}
    Now, by a telescoping bound using the semi-triangle inequality (Lemma \ref{lem:semitriang}):
    \begin{align*}
       \big\| T^*  E_{\cyl{x}} T - E_{L\cyl{x}}\big\|_{\HS}^2  \leq O(M) \sum_{j \in F_{M-1}}\left\| H_{<j} (T^*B_j T - B_{j+1}) H_{>j}\right\|^2_{\HS}\;,
    \end{align*}
where $H_{<j} = \prod_{i=-M+1}^{j-1}\frac{1}{2}(I+(-1)^{x_i}T^*B_iT)$ and $H_{>j} = \prod_{i=j+1}^{M-1}\frac{1}{2}(I+(-1)^{x_i}B_{i+1})$.
Summing over all $x \in \{0,1\}^{F_{M-1}}$, then changing order of summation and using the fact that $E$ is a PVM and Pythagoras, we obtain:
    \begin{align*}
        \sum_{x\in\{0,1\}^{F_{M-1}}} \big\| T^*  E_{\cyl{x}} T - E_{L\cyl{x}}\big\|_{\HS}^2  &= O(M) \sum_{j \in F_{M-1}}\left\|  (T^*B_j T - B_{j+1}) \right\|^2_{\HS} \\
        &= O(M^2\epsilon') \;,
    \end{align*}
    where in the final line we use the standing assumption~\eqref{eq:stand_ass3}.
    The approximate invariance of $\mu$ then follows from Lemma \ref{lem:PVM-to-meas}.
    \end{proof}

\subsection{Tower decompositions for approximately invariant measures}
\label{sec:tow_decomp_main_proof}
As was mentioned in the introduction, the cornerstone of our proof relies on a tower decomposition type result for the full shift.
This is inspired by classical Kakutani-Rokhlin decomposition lemmas. 
Informally, given an "approximately invariant" measure $\mu$ on $X = \{0,1\}^{\mathbb{Z}}$, the proposition provides a partition of $X$ into clopen towers of the full shift up to small error, with a bound on the complexity of the clopen sets appearing in these towers, such that moreover each tower is either approximately invariant or sufficiently high. 
Additional difficulties not appearing in classical considerations are the presence of periodic points, bounding the complexity of the constructed towers, as well as having only \emph{approximately invariant} measures.



\begin{proposition}[Tower decomposition]\label{prop:decomp}
Let $t\in\mathbb{N}$ and $0<\upsilon, \delta, \eta \leq 1/2$ be parameters. 
Let $\ell_0=\lceil C_1 t\log(1/\upsilon)/\delta\rceil$ and let
\[
    M_0=C_0 t\big(t\ell_0+t^3+t\log^*(\ell_0+2)\big),
\]
where $C_0,C_1>0$ are universal constants specified in the proof.
Let $\mu$ be a probability measure on $X = \{0,1\}^{\mathbb{Z}}$ which is $(M_0,\eta)$-invariant. 
Then there exists an $M_0$-definable subset $e\subseteq X$ with $\mu(e) = O(t^6(\upsilon+\delta+\eta))$ and a partition of $X\backslash e$ into a finite collection $\mathscr T$ of disjoint clopen towers such that:
\begin{enumerate}
\item  \textbf{(Complexity upper bound)} Each tower appearing in $\mathscr T$ is $M_0$-defined.
\item \textbf{(Height versus closedness dichotomy)} For each tower $\Tow(b,j)$ appearing in $\mathscr T$, either $j < t$ and $\Tow(b,j)$ is $\delta$-closed or $t\leq j < 6t+1$.
\item \textbf{(Complexity lower bound)} For each tower $\Tow(b,j)$ appearing in $\mathscr T$, the set $\pi_{\max(j,t)}(b)$ is a singleton.
\end{enumerate}
\end{proposition}

The crucial conclusion appearing above is the second condition, which asserts that each tower in the decomposition is either sufficiently high (has height at least $t$), or is sufficiently periodic. 
The proposition is proven in Section~\ref{sec:decomp}.

In our setting, since the measure $\mu$ induced by the PVM $E$ is $(M-1, \eps'')$ invariant by Lemma \ref{lem:mu-invariance}, we can apply Proposition \ref{prop:decomp} with parameters $t, \upsilon, \delta$ to be chosen later so that $M-1>M_0$, where $M_0$ is specified in Proposition \ref{prop:decomp}, and $\eta = \eps''$.
From here on, we fix the resulting tower decomposition $\mathscr T$.

\subsection{From tower decompositions to true representations}
\label{sec:tow_to_rep_main_proof}
With the decomposition $\mathscr T$ of the shift space $X$ to disjoint clopen towers at hand, we  can construct approximate closed projection towers (in the sense of Definition \ref{def:approx-proj-tower}) using the projection valued measure $E$:
For a clopen tower $\Tow(b,j) \in \mathscr T$, the sequence $(E_b,\dots, E_{L^{j-1}b}; T)$ will be an approximate closed projection tower.
The fact that the constructed towers are approximately closed will follow from the height versus closedness dichotomy guaranteed by Proposition \ref{prop:decomp}.
Then, by applying stability for approximate closed projection towers (Lemma \ref{lem:lowd}), we will be able to construct a nearby representation of the lamplighter group.

We start by quantifying how much $(E_b,\dots, E_{L^{j-1}b}; T)$ behaves like an approximate tower, by the following lemma.

\begin{lemma}\label{lem:clb}
For any tower $\tau=\Tow(b,j) \in \mathscr T$ let 
\begin{equation}\label{eq:def-deltab}
\delta^\tau_1 \,=\, \sum_{i=0}^{j-1} \big\| T^*  E_{L^i b} T - E_{L^{i+1}b} \big\|_{\HS}^2 \;.
\end{equation}
Then we have
\begin{equation}\label{eq:clb-0b}
\sum_{\tau \in\mathscr T}\delta^\tau_1 = O(\eps'')\;.
\end{equation}
\end{lemma} 

\begin{proof}
   Since the sets $L^i b$, where $\Tow(b,j) \in \mathscr T, i < j$, are all pairwise disjoint and $(M-1)$-defined, the sets $\pi_{M-1}(L^ib) \subset \{0,1\}^{F_{M-1}}$ are all pairwise disjoint as well.
   Therefore we have:
   \begin{align*}
    \sum_{\tau=\Tow(b,j) \in \mathscr T}\delta^\tau_1 &= \sum_{\Tow(b,j) \in \mathscr T}\sum_{i=0}^{j-1} \big\| T^*  E_{L^i b} T - E_{L^{i+1}b} \big\|_{\HS}^2 \\
        &\leq 4\sum_{\Tow(b,j) \in \mathscr T}\sum_{i=0}^{j-1}\sum_{x\in \pi_{M-1}(L^i b)} \big\| T^*  E_{\cyl{x}} T - E_{L\cyl{x}} \big\|_{\HS}^2 \\
        &\leq 4 \sum_{x\in\{0,1\}^{F_{M-1}}} \big\| T^*  E_{\cyl{x}} T - E_{L\cyl{x}} \big\|_{\HS}^2 \;,
   \end{align*}
where we applied Lemma~\ref{lem:p-ortho} below in the second line.
Together with approximate equivariance invariance of $E$ (Lemma~\ref{lem:mu-invariance}), this proves the lemma. 
\end{proof}

\begin{lemma}\label{lem:p-ortho}
Suppose $\{ P_x \}_x$ (resp.\ $\{ Q_x \}_x$) is a collection of pairwise orthogonal projections indexed by elements $x\in\mathcal{X}$, for some finite set $\mathcal{X}$. Let $b_1,\ldots,b_t \subseteq \mathcal{X}$ be disjoint subsets. Then 
\begin{equation}
     \sum_i \Big\| \sum_{x\in b_i} P_x - \sum_{x\in b_i} Q_x \Big\|_{\HS}^2 \leq 4 \sum_x \big\|P_x-Q_x\big\|_{\HS}^2\;.
\end{equation}
\end{lemma}

\begin{proof}
Write
\begin{align*}
    \sum_i \Big\| \sum_{x\in b_i} P_x - \sum_{x\in b_i} Q_x \Big\|_{\HS}^2 
    &=     \sum_i \Big\| \sum_{x\in b_i} P_x(P_x-Q_x) + \sum_{x\in b_i} (P_x-Q_x)Q_x \Big\|_{\HS}^2 \\
    &\leq 2    \Big( \sum_i \Big\| \sum_{x\in b_i} P_x(P_x-Q_x) \Big\|_{\HS}^2 + \sum_i \Big\| \sum_{x\in b_i} (P_x-Q_x)Q_x \Big\|_{\HS}^2\Big) \\
    &\leq  2    \Big( \sum_x \big\|P_x-Q_x \Big\|_{\HS}^2 + \sum_x \big\| P_x-Q_x \big\|_{\HS}^2\Big) \;.
\end{align*}
Here for the first line we used that $P_x, Q_x$ are projections, for the second line Lemma \ref{lem:semitriang}, and for the third line that the $P_x$ and $Q_x$ are pairwise orthogonal respectively and $0\leq P_x,Q_x \leq I$.
\end{proof}

The following claim will be crucial to construct a lamplighter representation out of a tower $\tau \in \mathscr T$, by identifying an element $x^\tau \in \{0,1\}^{[0,j-1]}$ which is suitable for Lemma \ref{lem:tower_to_rep}.

\begin{claim}\label{claim:b-B}
Let $\tau = \Tow(b,j) \in  \mathscr T$, then for every $0 \leq i \leq j-1$ we have 
$$ B_0 E_{L^ib} = (-1)^{x_i^\tau} E_{L^ib}\;,$$
for some $x_i^\tau \in \{0,1\}$.
\end{claim}

\begin{proof}
By the complexity lower bound guarantee in Proposition \ref{prop:decomp}, we have that $\pi_j(b)$ is a singleton.
Thus, for any $0\leq i< j$ all elements of  $\cyl{\pi_j(b)}$ have the same $(-i)$-th coordinate.
Thus, all elements of $L^i \cyl{\pi_j(b)}$ have the same $0$-th coordinate, denote it by $x_i^\tau \in \{0,1\}$.
Consequently, $E_{L^i \cyl{\pi_j(b)}}$ is included in the $(-1)^{x_i^\tau}$-eigenspace of $B_{0}$.
Since  $b\subseteq \cyl{\pi_j(b)}$, the same is true for $E_{L^i b}$, as claimed. 
\end{proof}

To quantify how approximately closed the projection towers constructed are, we will divide according to the Height versus closeness dichotomy in Proposition \ref{prop:decomp}.
To this end, denote by $\mathscr T_{< t} \subset \mathscr T$ the subset of towers of height less than $t$ ("short" towers), and $\mathscr T_{\geq t}$ the subset of towers of height at least $t$ ("high" towers).
We start with constructing representations from high towers.
Going back to the observation pointed out after Definition \ref{def:approx-proj-tower}, the approximate projection towers coming from high towers will automatically be approximately closed, a crucial fact which will be used in the following Lemma.

\begin{lemma}[Representations from high towers]\label{lem:tower-long}
Let $\tau = \Tow(b,j) \in \mathscr T_{\geq t}$, and let $P_i = E_{L^ib}$ for $0\leq i <j$.
Abuse notation and also denote $\tau = (P_0, \dots, P_{j-1}; T)$.
Then $\tau$ is an approximate $(\delta^\tau_1,\delta_2)$-closed projection tower, where $\delta_2 = O(\tr(P_\tau)/t + j\cdot \delta^\tau_1)$.
Further, there exists a unitary representation $\rho_{\tau } : \Gamma \to U(P_\tau \mathcal H)$ such that:

\begin{align*}
\| P_\tau B_0 P_{\tau} - \rho_{\tau}(a) \|_{\HS}^2 &= O(j\delta^\tau_1),\\
    \| P_\tau T P_\tau- \rho_{\tau}(t^{-1})\|_{\HS}^2 & = O(j^2 \delta^\tau_1 + \tr(P_\tau)/t).
\end{align*}
\end{lemma}

\begin{proof}

We start by bounding $\delta_2$. Note that by the proof of Lemma \ref{lem:lowd} (see Claim \ref{claim:p-bound}), we have $\tr(P_i) \leq \tr(P_\tau)/j+O(j\delta^\tau_1) \leq \tr(P_\tau)/t+O(j\delta^\tau_1)$ for all $i <j$. 
Consequently, we have by the semi-triangle inequality (Lemma \ref{lem:semitriang}):
\begin{align*}
    \| T^*  P_{j-1} T  - P_{0}\|_{\HS}^2 \leq 2\big(\tr(P_{j-1})+\tr(P_0)\big) \leq \tr(P_\tau)/t + O(j\delta^\tau_1)
\end{align*}
We can now apply stability for approximate closed towers (Lemma~\ref{lem:lowd}) to $\tau$, and obtain a (true) closed projection tower $(P_0', \dots, P_{j-1}';T')$ with $P_i'\leq P_i$ for all $i$,  which is $(\eps_1,\eps_2)$-close to $\tau$, where
$$\eps_1 = O(j\delta^\tau_1) \;,\quad \eps_2 = O(j^2 \delta^\tau_1 + \tr(P_\tau)/t)\;, $$
$T'\in U(\mathcal H)$ acting as identity on $I - \sum_{i=0}^{j-1} P_i'$. 
By Claim \ref{claim:b-B}, there is $x^\tau \in \{0,1\}^{[0,j-1]}$ such that
$$ B_0 P_i = (-1)^{x_i^\tau} P_i \quad\text{for all $0\leq i<j$ }. $$
Consequently, we have $P_\tau B_0P_\tau = B_0P_\tau = \sum_{i=0}^{j-1} (-1)^{x^\tau_i} P_i$.
By using Lemma \ref{lem:tower_to_rep} on the closed projection tower $(P_0',\dots,P_{j-1}' ; P_\tau T')$, treated as a closed projection tower on $P_\tau \mathcal H$, and $x^\tau$, we obtain a unitary representation $\rho_{\tau} : \Gamma \to U(P_\tau \mathcal H)$ where
\begin{align*}
    \rho_{\tau}(a) &= \sum_{i=0}^{j-1} (-1)^{x_i^\tau} P_i' + \Big(P_\tau - \sum_{i=0}^{j-1} P_i'\Big)\;, \\
    \rho_{\tau}(t)&= P_{\tau}T'^*\;.
\end{align*}
It follows by Pythagoras and the definition of $(\eps_1, \eps_2)$-closeness (see~\eqref{eq:lowd-0b}) that
\begin{align*}
    \| P_\tau B_0 P_{\tau} - \rho_{\tau}(a) \|_{\HS}^2 &= \big\| \sum_{i=0}^{j-1} ((-1)^{x_i^\tau}-1) (P_i - P_i')\big\|_{\HS}^2  + \| P_\tau - \sum_{i=0}^{j-1} P_i' \|_{\HS}^2 \\
    &= O(\eps_1) = O(j\delta^\tau_1)\;.
\end{align*}
Similarly, again by definition of $(\eps_1, \eps_2)$-closeness (see~\eqref{eq:lowd-0c}), we have
\begin{align*}
    \| P_\tau T P_\tau- \rho_{\tau}(t^{-1})\|_{\HS}^2 = O(\eps_2) = O(j^2 \delta^\tau_1 + \tr(P_\tau)/t)\;.
\end{align*}
\end{proof}

Next we consider the case of short towers, we first estimate the approximate closure of the associated approximate projection towers.

\begin{lemma}\label{lem:clb2}
For any tower $\tau=\Tow(b,j) \in \mathscr T_{< t}$, let 
\begin{equation}\label{eq:def-deltab2}
\delta^\tau_2 \,=\, \big\| T^*  E_{L^{j-1}b} T - E_{b} \big\|_{\HS}^2 \;.
\end{equation}
Then we have
\begin{equation}\label{eq:clb-0b2}
\sum_{\tau\in\mathscr T_{< t}}\delta^\tau_2 = O(\delta + t^2\eps'')\;.
\end{equation}
\end{lemma} 
\begin{proof}
   For $j< t$, we will denote by $\mathscr T_{=j} \subset\mathscr T$ the subset of towers of height \emph{exactly} $j$.
   Let us fix $\tau=\Tow(b,j) \in \mathscr T_{< t}$. 
   Firstly, by the height versus closedness dichotomy ($(ii)$ in Proposition \ref{prop:decomp}), $\Tow(b,j)$ is $\delta$-closed, so that $\mu(L^jb\cap b) \geq (1-\delta)\mu(b) $.
   Secondly, note that by definition of $\delta_1^\tau$ in Lemma \ref{lem:clb} and the semi-triangle inequality, we have
   \begin{align*}
         \delta^\tau_2  \leq 2\| E_{L^jb } - E_b \|_{\HS}^2 + 2\delta_1^b &= O\left(\mu(L^j b)   + \mu(b) - 2\mu(L^jb \cap b) + \delta_1^\tau \right) \\
         &= O(\mu(L^jb \triangle b) + \delta_1^\tau)\;. 
   \end{align*}
   For each $0 \leq j < t$, since all the bases $b$ arising in $\Tow(b,j) \in \mathscr T_{=j}$ are all disjoint, $L^jb,b$ are both $M-1$-defined and $\mu$ is $(M-1, \eps'')$-invariant, we can apply Lemma \ref{lem:approx_inv_measures} (proved in Section 5) to the collection of bases arising from $\mathscr T_{=j}$ to deduce
   $$ \sum_{\Tow(b,j) \in \mathscr T_{= j}} \mu(L^jb \triangle b)  \leq O(\delta \mu(W_j) + j\eps'')\;, $$
   where $W_j$ is the union of all bases in $\mathscr T_{=j}$. Consequently:
   \begin{align*}
    \sum_{\Tow(b,j) \in \mathscr T_{< t}}\delta^b_2 &\leq O\left(\sum_{j=0}^{t-1}\left(\delta \mu(W_j)+j\eps'' \right) + \sum_{\Tow(b,j) \in \mathscr T_{< t}}\delta_1^b \right) \leq  O(\delta + t^2\eps'')\;.
   \end{align*}
   Here we also used $\sum_{\Tow(b,j) \in \mathscr T_{< t}}\delta_1^b = O(\eps'')$ provided by Lemma~\ref{lem:clb}, and the fact that the sets $W_j$ are disjoint.
\end{proof}

We can now construct representations from short towers.

\begin{lemma}[Representations from short towers]\label{lem:tower-short}
Let $\tau = \Tow(b,j) \in \mathscr T_{< t}$ and let $P_i = E_{L^ib}$ for $0\leq i <j$.
Abuse notation and also denote $\tau = (P_0, \dots, P_{j-1}; T)$.
Then $\tau$ is an approximate $(\delta^b_1,\delta_2^b)$-closed projection tower, and there exists a unitary representation $\rho_{\tau } : \Gamma \to U(P_\tau \mathcal H)$ such that
\begin{align*}
\| P_\tau B_0 P_{\tau} - \rho_{\tau}(a) \|_{\HS}^2 &= O(j\delta^b_1)\;,\\
    \| P_\tau T P_\tau- \rho_{\tau}(t^{-1})\|_{\HS}^2 & = O(j^2 \delta^b_1 + \delta_2^b)\;.
\end{align*}
\end{lemma}

\begin{proof}
The proof is identical to the proof of Lemma~\ref{lem:tower-long}. The only difference is that $\delta_2^b$ is defined so that $\tau$ is by definition a $(\delta_1^b,\delta_2^b)$-closed projection tower.
\end{proof}

Finally, we can give the proof of Theorem \ref{thm:main}. We will need a final easy consequence of the semi-triangle inequality regarding approximate projection towers:
\begin{claim}
    \label{claim:approx_inv_supp}
Let $\tau = (P_0, \dots, P_{j-1}; R)$ be an approximate $(\delta_1,\delta_2)$-closed projection tower, then $P_\tau$ is approximately invariant under $R$: 
$$\| P_\tau R - R P_\tau \|_{\HS}^2= \| R^* P_\tau R - P_\tau \|_{\HS}^2 \leq j\delta_1 + \delta_2.$$
\end{claim}

\begin{proof}[Proof of Theorem \ref{thm:main}]

Since the tower decomposition $\mathscr T = \mathscr{T}_{<t} \sqcup \mathscr T_{\geq t}$ is disjoint and covers $X$ together with the error set $e$, we have an orthogonal decomposition 
\[\mathcal H = \bigoplus_{\tau \in \mathscr T} P_\tau \mathcal H \oplus E_e\mathcal H\;,\]
where $P_\tau$ is as in Lemmas \ref{lem:tower-long} and \ref{lem:tower-short}. 
It follows that
\[
    \rho = \bigoplus_{\tau \in \mathscr T} \rho_\tau \oplus 1_{E_e\mathcal H}
\]
is a unitary representation of $\Gamma$ on $\mathcal H$.
It follows from Pythagoras and Lemmas \ref{lem:tower-short} and \ref{lem:tower-long} that
\begin{align*}
    \| \rho(a) - B_0\|_{\HS}^2 &= \sum_{\tau \in \mathscr T} \| P_\tau(\rho_\tau(a) - B_0)P_\tau\|_{\HS}^2  + 2\mu(e) \leq \sum_{\Tow(b,j) \in \mathscr T} O(j\delta_1^b) + 2\mu(e) \\
    &\leq O(t\eps'' + t^6(\upsilon+ \delta+\eps'') ) \\
    &= O(t^6(\upsilon+ \delta+\eps''))\;,
\end{align*}
where in the last line we used Lemma \ref{lem:clb} and the bound on $\mu(e)$ appearing in Proposition \ref{prop:decomp}, as well the fact that all towers in $\mathscr T$ have height less than $6t+1$.
To bound the distance between $\rho(t^{-1})$ and $T$, by Lemmas \ref{lem:tower-short} and \ref{lem:tower-long} and Claim \ref{claim:approx_inv_supp} we have:
\begin{align*}
     \sum_{\tau \neq \tau' \in \mathscr T} \|P_\tau T P_{\tau'} \|_{\HS}^2 &= \sum_{\tau \neq \tau' \in \mathscr T, \tau \in \mathscr T_{<t} } \|P_\tau T P_{\tau'} \|_{\HS}^2 + \sum_{\tau \neq \tau' \in \mathscr T, \tau \in \mathscr T_{\geq t} } \|P_\tau T P_{\tau'} \|_{\HS}^2 \\
     &\leq \sum_{\tau \in \mathscr T_{<t}} O(t\delta_1^b + \delta_2^b) + \sum_{\tau \in \mathscr T_{\geq t}} O(t\delta_1^b + \tr(P_\tau)/t)\;,
\end{align*}
where we also used that $P_\tau P_{\tau'}=0$ for $\tau \neq \tau'$. 
It follows from Pythagoras, Lemmas \ref{lem:tower-short} and \ref{lem:tower-long} again that
\begin{align*}
    \| \rho(t^{-1}) - T\|_{\HS}^2 &= \sum_{\tau \in \mathscr T_{<t}} \| P_\tau(\rho_\tau(t^{-1}) - T)P_\tau\|_{\HS}^2 + \sum_{\tau \in \mathscr T_{\geq t}} \| P_\tau(\rho_\tau(t^{-1}) - T)P_\tau\|_{\HS}^2 \\
    & +\sum_{\tau \neq \tau' \in \mathscr T} \|P_\tau T P_{\tau'} \|_{\HS}^2 + 2\mu(e) \\
    &\leq \sum_{\tau \in \mathscr T_{<t}} O(t^2\delta_1^b + \delta_2^b) + \sum_{\tau \in \mathscr T_{\geq t}} O(t^2\delta_1^b + \tr(P_\tau)/t)  + 2\mu(e) \\
    &\leq O(t^2\eps'' + \delta + t^2\eps'' + t^2\eps'' + 1/t +  t^6(\upsilon + \delta+\eps''))  \\
    &\leq O(1/t + t^6(\upsilon + \delta+\eps'')).
\end{align*}
Choose $t=\lceil C\kappa^{-2}\rceil$ and $\delta=\upsilon=c\kappa^{14}$, and require $\eps''\leq c\kappa^{14}$, with the constants chosen so that the last two displayed bounds are at most $\kappa^2/4$. For these choices,
\[
    \ell_0=O\big(\kappa^{-16}\log(2/\kappa)\big),
    \qquad
    M_0=O\big(\kappa^{-20}\log(2/\kappa)\big)
\]
by Proposition~\ref{prop:decomp}. Taking
\[
    M\geq M_0,
    \qquad
    \eps\leq c\,\kappa^7/M^2
\]
gives $\eps''=O(M^4\eps^2)\leq c\kappa^{14}$ by Lemma~\ref{lem:mu-invariance}. 
 Hence
\[
    \|\rho(a)-B_0\|_{\HS}\leq \kappa/2,
    \qquad
    \|\rho(t^{-1})-T\|_{\HS}\leq \kappa/2.
\]
Setting
\[
    \widetilde A=\rho(a),
    \qquad
    \widetilde T=\rho(t^{-1}),
\]
this proves the theorem after adjusting the universal constants for the pre-processing steps (described at the start of Section~\ref{sec:proof-main}).
\end{proof}

\section{Efficient marker sets for general group actions} \label{sec:marker-efficient}

To construct a tower decomposition as in Proposition \ref{prop:decomp}, we will need the notion of \emph{marker sets}, an important tool appearing in ergodic theory, topological dynamics and descriptive graph combinatorics  \cite{GJ, schneider2013locally, GJKS, naryshkin2024urp, grebik2025descriptive}.
Generally speaking, a marker set $Z$ is a subset of a group action $\Gamma \acts X$ which is both sufficiently independent, in the sense that $gZ \cap Z = \emptyset$ for all $g\neq 1$ in some finite subset $F \subset \Gamma$,  as well as sufficiently saturating, in the sense that finitely many $\Gamma$-translates of $Z$ cover $X$.
In the combinatorial sense, marker sets are simply \emph{maximal independent sets} in the Schreier graph of the action, an interpretation which will be important for us.
The goal of this section is Proposition \ref{prop:marker_general}, which guarantees existence of clopen marker sets for general group actions on zero-dimensional polish spaces, which are moreover \emph{effective}, in the sense that the marker is built out of a given input coloring.
With that said, we will only need Proposition \ref{prop:marker}, which is the specialization of Proposition \ref{prop:marker_general} to the full shift over $\mathbb Z$.

To achieve this, we appeal to recent results in the confluence of descriptive combinatorics and LOCAL distributed algorithms, largely following the surveys \cite{bernshteyn2022descriptive, grebik2025descriptive} and the foundational paper \cite{Bernshteyn_inv}, which we recommend for background.
Let $\Gamma$ be a countable discrete group, $X$ be a zero-dimensional Polish space together with a continuous $\Gamma$-action.
All the following statements hold equally well when $X$ is a standard Borel space, and $\Gamma \acts X$ in a Borel manner, by replacing the words "clopen"  and "continuous" by "Borel".
Given a finite symmetric subset $1 \in F \subset \Gamma$, denote by $\mathscr G^F_X \subset X \times X$ the Schreier graph of the action:
$$ \mathscr{G}_X^F = \{(x, g\cdot x) | \; x \in X, g \in F\setminus \{1\} \}.$$
Since $F$ is symmetric, we may treat $\mathscr G_X^F$ as an undirected graph on $X$.
For $Y\subset X$, we denote by $\mathscr G_Y^F \subset Y\times Y$ the induced graph on $Y$. If $Y$ is clopen, then $\mathscr G_Y^F$ is a \emph{locally finite Borel graph}, with a maximum degree bound $\Delta = |F|$ (see \cite{bernshteyn2022descriptive}).
Further, if $Y$ is compact open or $Y = X$, then $\mathscr G_Y^F$ is a \emph{topological graph}, in the sense of \cite[Definition 2.12]{Bernshteyn_inv}.

\begin{definition}[Continuous colorings] \label{def:cont_col}
    For $1 \in F \subset \Gamma$,  $Y\subset X$ clopen as above and $N \in \N$, a continuous map $\varphi: Y \to [N]$ is said to be a \emph{\textbf{$F$-proper continuous coloring}} if it is a proper coloring of $\mathscr G_Y^F$.
\end{definition}

A tautological but important observation is that $\varphi: Y \to [N]$ is a $F$-proper coloring if and only if for each $i \in [N]$, $\varphi^{-1}(i) \subset Y$ is \emph{$F$-independent}, meaning $g\cdot \varphi^{-1}(i) \cap \varphi^{-1}(i) = \emptyset$ for all $i \in [N], g \in F \setminus \{1\}$.
The following terminology will be a useful extension of definability introduced in Section \ref{sec:Ber}, allowing to work with general group actions.

\begin{definition}[Definability of clopen sets]\label{def:definability-clopen}
For $Y \subset X$ clopen as above, and $1 \in T \subset \Gamma$ finite, we say a set $A \subset Y$ is \textbf{\emph{$T$-defined}} with respect to a finite collection of clopen sets $B_1, \dots, B_N \subset Y$ if $A$ is in the algebra of clopen sets generated by $\{ (g\cdot B_i) \cap Y|\; g \in T, i \leq N\}$.

Similarly, if $\varphi: Y \to [N]$ is a continuous map, we say $A$ is $T$-defined with respect to $\varphi$ if it is $T$-defined with respect to the pre-images of $\varphi$.
A continuous map $\psi: Y \to [M]$  is $T$-defined with respect to $\varphi$ if all its pre-images are $T$-defined with respect to $\varphi$.
\end{definition}

Note that in the case of $X =Y = \{0,1\}^\mathbb Z$, $\varphi = \pi_0$ and the action by the full shift, $n$-definability introduced in Section \ref{sec:Ber} is the same as $F_n$-definability introduced in Definition \ref{def:definability-clopen}.

We will need the following result of Linial \cite{Linial} (see also \cite[Corollary 3.24]{BarenboimElkin2013}), regarding the round complexity of $O(\Delta^2)$-graph coloring in distributed LOCAL algorithms.
We refer to \cite[Section 3.10]{BarenboimElkin2013} for two beautiful combinatorial proofs and more background on the rich subject of distributed graph algorithms.

\begin{lemma}[Linial \cite{Linial}]\label{lem:loca_alg_col_red}
There exists a deterministic single-round LOCAL algorithm that given a graph $G$ of degree $\Delta$ and a proper $N$-coloring of it, finds a proper $5\lceil \Delta^2 \log(N)\rceil$-coloring of $G$.

Furthermore, there is a LOCAL algorithm that iteratively re-colors $G$ in $r(N)=\log^*(N)+ O(1)$-rounds and returns an $O(\Delta^2)$-coloring of it.
\end{lemma}

Note that the constants in the big $O$ notation above are universal and independent of $\Delta$, and $\log^*$ denotes the iterated logarithm function: $\log^*(N)$ is the number of times $\log$ needs to be applied iteratively to $N \in \mathbb N$ to be less than or equal to $2$. 
The ``Furthermore'' part almost follows by iterating the first part of the lemma, with an additional round required to reduce from $O(\Delta^2\log\Delta)$ to $O(\Delta^2)$ colors. 
Note that unlike general LOCAL algorithms, that receive as input a unique identifiers/bijective coloring for the graph (see \cite[Definition 2.8]{Bernshteyn_inv}), Linial's algorithm merely assumes that the input coloring is \emph{proper}.
This will be important in the subsequent use of Lemma \ref{lem:loca_alg_col_red}.


The following is an efficient version of Schneider and Seward's marker lemma \cite[Lemma 3.1]{schneider2013locally}, due to Petr Naryshkin \cite{Petr_arg}.

\begin{proposition}[Efficient marker lemma, general version \cite{Petr_arg}]\label{prop:marker_general}
Let $\Gamma \acts X$ be an action of a countable group on a zero-dimensional polish space.
Let $1 \in F \subset \Gamma$ be a finite symmetric subset, $Y\subset X$ a compact open subset or $Y= X$. 
Assume  $\varphi: Y \to [N]$ is a given continuous $F$-proper coloring for some $N \in \mathbb N$.
Then there exists a clopen subset $Z \subset Y$ which satisfies:
\begin{enumerate}
    \item $Z$ is $F$-independent.
    \item The collection $gZ, g \in F$ covers $Y$.
    \item $Z$ is $B_F(r(N) + O(|F|^2))$-defined with respect to $\varphi$. 
\end{enumerate}
\end{proposition}

The proposition is specialized to the full shift setup in the next section to derive Proposition~\ref{prop:marker}, yielding marker sets of polynomial complexity.

To prove Proposition \ref{prop:marker_general}, we need the fact that efficient deterministic local graph algorithms yield continuous colorings of (topological) Borel graphs.
This follows from the proof of a recent result of Bernshteyn \cite[Theorem 2.13]{Bernshteyn_inv} (see also  \cite[Theorem 2.11]{bernshteyn2022descriptive} ), together with Linial's efficient LOCAL coloring algorithm (Lemma \ref{lem:loca_alg_col_red}) for $\Delta = |F|$.
The idea is to use the assumed $F$-proper coloring $\varphi:Y \to [N]$ as a starting point for Linial's algorithm, which then progressively reduces the number of colors used.
After $r(N)$ local iterations this yields an  $F$-proper coloring with only $O(\Delta^2)$-many colors. Moreover, by the way it is constructed this coloring is $T$-defined with respect to $\varphi$, where $T = B_F(r(N))$. Formally, we have the following.


\begin{lemma}[Efficient color reduction]\label{lem:col_red}
Let $\Gamma \acts X$ be an action of a countable group on a zero-dimensional polish space.
Let $1 \in F \subset \Gamma$ be a finite symmetric subset, $Y\subset X$ a compact open subset or $Y= X$. 

Assume that there exists  $N \in \N$ 
and a continuous $F$-proper coloring $\varphi: Y \to [N]$. 
Then there exists a continuous $F$-proper coloring $\psi: Y \to [O(|F|^2)]$ which is $T=B_F(r(N))$-defined with respect to $\varphi$.
\end{lemma}

\begin{proof}
As mentioned earlier, there exists a deterministic distributed LOCAL algorithm that finds a $O(\Delta ^2)$-coloring of a graph of maximum degree $\Delta$ in $r(N)$ rounds, assuming an input proper $N$-coloring  (Lemma \ref{lem:loca_alg_col_red}).
Recall the notions of \cite[Theorem 2.13]{Bernshteyn_inv}, and consider the collection $\texttt{G}$ of all isomorphism classes of finite structured graphs \cite[Definition 2.1]{Bernshteyn_inv} with maximum degree at most $\Delta = |F|$, together with the local coloring problem $\Pi$ of proper coloring \cite[Example 2.7]{Bernshteyn_inv}. 
Observe that in the proof presented in \cite[Section 4]{Bernshteyn_inv}, instead of using Kechris-Solecki-Todorcevic to find an initial unique identifiers Borel proper coloring $c:Y \to [N]$, 
in our context it is sufficient to start from the assumed coloring $\varphi:Y \to [N]$, which is by assumption $F$-proper.
We may then apply Linial's algorithm to $\varphi$ for $r(N)$ rounds to iteratively construct a coloring $\psi:Y \to [O(|F|^2)]$ of $\mathscr G_Y^F$ which is $T$-defined with respect to $\varphi$.
 The $T$-definability of $\psi$ follows from the observation that the map taking $x \in Y$ to the isomorphism class of the  rooted finite graph $\mathscr G_{T\cdot x \cap Y}^F$ endowed with the coloring $\varphi |_{T \cdot x \cap Y}$ is $T$-defined with respect to $\varphi$ (see \cite[Lemma 4.1]{Bernshteyn_inv}). 
\end{proof}

Having this at hand, we can now prove Proposition \ref{prop:marker_general} by using a greedy maximal independent set algorithm, as in \cite[Lemma 3.1]{schneider2013locally}.

\begin{proof}[Proof of Proposition \ref{prop:marker_general}]
By applying Lemma \ref{lem:col_red}, we obtain an $F$-proper coloring $\psi: Y \to [O(|F|^2)]$ which is $T$-defined with respect to the assumed coloring $\varphi$.
Consider the following iterative greedy algorithm with $M=O(|F|^2)$ many steps: Set $Z_1 = \psi^{-1}(1)$, and for each $1 < i \leq M$, set
$$Z_i = Z_{i-1} \cup  \Big(\psi^{-1}(i) \setminus \bigcup_{g \in {F}} g\cdot  Z_{i-1} \Big)\;.$$
We claim that setting $Z = Z_{M} \subset Y$ we get the desired set.
The proof that $Z$ possesses properties $(i), (ii)$ in the statement of the proposition follows from the proof of \cite[Lemma 3.1]{schneider2013locally}, together with the fact that $\psi^{-1}(i)$ is $F$-independent for all $i$.
Recall that $T = B_F(r(N))$,
we will finish the proof by showing inductively on $i\leq M$ that $Z_i$ is $B_F(r(N)+i-1)$-defined with respect to $\varphi$.
Indeed, we have that $Z_1$ is $T$-defined with respect to $\varphi$ as guaranteed by Lemma \ref{lem:col_red}.
For $1< i \leq M$, since $\psi^{-1}(i)$ is $T$-defined with respect to $\varphi$, and  $\bigcup_{g \in {F}} g\cdot  Z_{i-1}$ is $F$-defined with respect to $Z_{i-1}$, we see that $Z_i$ is $F \cdot B_F(r(N)+i-2) \subset B_F(r(N) + i-1)$-defined with respect to $\varphi$.
\end{proof}

\section{Tower decomposition for approximately invariant measures}
\label{sec:decomp}

In this section we give the proof of the main tower decomposition procedure (Proposition \ref{prop:decomp}).
The proof will start by covering approximately periodic sequences by approximately closed towers, then proceed to cover the complement by high towers constructed via the polynomial marker lemma.
We first set up notation to be used throughout this section.

Recall notation and terminology from Section \ref{sec:Ber} regarding the full shift $L:X \to X$, where $X = \{0,1\}^\Z$.
Fix parameters $t, \delta, \eta, \upsilon$, and let $\mu$ be a given probability measure on $X$.
After the marker set is constructed, we will choose the locality parameter $M$ in \eqref{eq:M-poly} and assume that $\mu$ is $(M,\eta)$-invariant.

Recall $F_t = [-t,t]$ and $X_{per}(t)$ are the periodic sequences with period at most $t$, that is:
$$ X_{per}(t) = \{ x \in X | \exists i \in F_t \setminus \{0\} \; L^i x = x \}\;.$$
Now let $X_{aper}(t) = X \setminus X_{per}(t)$ be the complement; we will refer to its elements as \textbf{\emph{$F_t$-aperiodic sequences}}.

It will be important to consider clopen versions of $X_{per}(t)$ and $X_{aper}(t)$. 
For this, we introduce the following clopen sets:

\begin{definition}[Approximately aperiodic sequences] \label{def:approx_per}
Given $t,\ell \in \mathbb N$, let
\begin{align*}
X_{aper}^\ell(t) &=  \{ x \in X | \; \cyl{\pi_\ell(x)} \text{ is $F_t$-independent}\} \subset X\;,\\
X_{per}^\ell(t) &=  X \setminus X_{\aper}^\ell(t) \;.
\end{align*}
\end{definition}
That is, $X_{aper}^\ell(t)$ are the sequences which are $\ell$-defined, and are "$F_t$-aperiodic" on the interval $F_\ell$, in the sense that the corresponding cylinder set is $F_t$-independent.

It follows from the definition that the sets $X^\ell_{\aper}(t), X^\ell_{\per}(t)$ are clopen and $\ell$-defined.
We will also denote by 
\[ Y^\ell_{\aper}(t) = \pi_\ell(X^\ell_{\aper}(t))\qquad\text{and}\qquad Y^\ell_{\per}(t) = \pi_\ell(X^\ell_{\per}(t))\]
the corresponding projections to $\{0,1\}^{F_\ell}$.
Observe that by definition, the collection $\mathscr B = \{ \cyl{y} \}_{y \in Y_{aper}^\ell(t)}$ consists of $\ell$-defined cylinder sets which cover $X_{aper}^\ell(t)$ and are $F_t$-independent, a crucial property that will be used later in the proof.
Note also that $ | \mathscr B | \leq 2^{2\ell+1}$.
%
%
%
%
%
%

\subsection{Covering the periodic sequences} \label{subsec:per}

\subsubsection{The minimal periodic extension}

The following notion is related to the ``orbit closing lemma'' for the full shift \cite[Proposition 8.6]{levit2023characters}, \cite{Par}, and the fact that it has dense periodic measures.

\begin{definition}[minimal periodic extension]\label{def:min_period}
Let $t \in \N$, and $x\in\{0,1\}^{F_t}$, let $\per(x) \in \N$ be the \textbf{\emph{minimal period}} of $x$ defined by:
$$ \per(x) = \min\{j \in \mathbb N | x \in \pi_t(X_{\per}(j))\}\;.$$
And the \textbf{\emph{minimal periodic extension}} $\ext(x) \in X$ to be the unique element of $\cyl{x} \cap X_{\per}(\per(x))$.
\end{definition}

It is easy to check that the minimal period of $x \in \{0,1\}^{F_t}$ is at most $2t+1$, and that the minimal periodic extension is well defined (that is, $\cyl{x} \cap X_{\per}(\per(x))$ is a singleton). 
We will also use the following notation to reduce clutter: 

\begin{definition}[extension and restriction]
Given $x\in\{0,1\}^{F_t}$, $\ell\in\mathbb{N}$ define $x^\ell\in\{0,1\}^{F_\ell}$ as follows:
\begin{enumerate}
    \item If $\ell \leq t$, we let $x^\ell = \pi_\ell(x) \in \{0,1\}^{F_\ell}$ be the restriction to $F_\ell$.
    \item If $t < \ell$, we let $x^\ell = \pi_\ell(\ext(x)) \in \{0,1\}^{F_\ell}$.
\end{enumerate}
\end{definition}

For example, if $x=100\in\{0,1\}^{[1]}$ then $\per(x)=3$ and $x^2=01001$, and if $x=101\in\{0,1\}^{F_1}$ then $\per(x)=2$ and $x^2=01010$.

\begin{lemma}[Properties of the minimal periodic extension]\label{lem:per_ext_properties}
The following statements hold:
\begin{enumerate}
    \item For $x \in \{0,1\}^{F_t}$,  $\per(x)$ is the \emph{maximal} number $j$ such that $\cyl{x}$ is $F_{j-1}$-independent, and that
    \item For $x \in \{0,1\}^{F_t}$, we have:
    \begin{equation}\label{eq:def-period}
    \per(x)=\min\big\{j\in\{1,\ldots,2t+1\}:\ (L^j\cyl{x})|_{[-t+j, t]} = \cyl{x} |_{[-t+j,t]}\big\}\;.
    \end{equation}
    \item If $r \geq 2t$, $y \in \{0,1\}^{F_r}$ is such that $\per(y) \leq t$ and $x = y^{2t}$, then $\per(x) = \per(y)$ and $y = x^r$.
\end{enumerate}
\end{lemma}
\begin{proof}
    The proofs of the first two claims are left for the reader. For part (iii), set $j=\per(y)$, note that $\per(x) \leq j$ by part (i), since $\cyl{x} \subset \cyl{y}$.
    Conversely, by part (ii), we have $(L^{j-1} \cyl{y})|_{[-r+j-1, r]} \neq \cyl{y} |_{[-r+j-1,r]}$, so that there exists $-r+j-1 \leq i \leq r$ such that $y_i \neq y_{i - (j-1)}$. 
    Write $i =qj + b$ for $0 \leq |b| < j$ such that $|qj| \leq |i|$.
    Since $y$ is the restriction of a $j$-periodic element of $\{0,1\}^{\Z}$, we have $y_b \neq y_{b-(j-1)}$ as well.
    However, $|b|, |b-(j-1)| \leq 2t$, so that $x_b \neq x_{b-(j-1)}$ and as such $\per(x) > j-1$ by part (ii) again. 
    Consequently, $\per(x) =\per(y)$ and $y = x^r$ follows readily. 
\end{proof}
\begin{remark}\label{rem:min_per}
    In what follows, it is useful to observe that the set $Y^\ell_{\per}(t)$ can be identified  using the minimal periodic extension as
    $$ Y^\ell_{\per}(t) = \big\{ y \in \{0,1\}^{F_\ell} | \; \per(y) \leq t \big\}\;.$$
\end{remark}

\subsubsection{Constructing periodic towers}

We can now formulate the key lemma for this subsection, which provides a procedure to cover $X^\ell_{\per}(t)$ by approximately closed towers. Recall that a tower $\Tow(b,j)$ is said to be \emph{$\delta$-closed} if $\mu(L^j b \cap b)\geq (1-\delta)\mu(b)$.

\begin{lemma} \label{lem:covering_per_seq}
Let $\frac12\ge \upsilon,\delta>0$ and $t$ an integer. Let $\ell\in\mathbb{N}$ such that $\ell\geq  C_1 t\log(1/\upsilon)/\delta$, where $C_1>0$ is a universal constant. Then there is a disjoint covering $\Xi_{\per} = E \sqcup \bigsqcup \Tow(b,j)$ of $X_{per}^\ell(t)$ consisting of an error set $E$ and a disjoint union of towers $\Tow(b,j)$ which satisfy:
\begin{itemize}
    \item $E$ is $\ell$-defined and $\mu(E) < \upsilon$.
    \item The towers $\Tow(b,j)$ are $\delta$-closed, $O(t\ell)$-defined and of height at most $6t+1$.
    \item  Each base $b$ of a tower $\Tow(b,j)$ is $\ell$-defined, and the projection $\pi_t(b) \in \{0,1\}^{F_t}$ is a singleton.
\end{itemize}
\end{lemma}

\begin{proof}
We first show a claim.

\begin{claim}\label{claim:62}
There is a universal constant $C_2>0$ such that the following holds. 
For any $r \geq 1$ and any $x\in\{0,1\}^{F_{r}}$ there is $r \leq \ell_x\leq C_2 r \log(1/\upsilon)/\delta$ such that either $\mu(\cyl{x^{\ell_x}})\leq \upsilon\mu(\cyl x)$ or $b=\cyl {x^{\ell_x}}$ is $F_{\per(x)-1}$-independent and $\Tow(b,\per(x))$ is a tower which satisfies $\mu(L^{\per(x)} b \cap b)\geq (1-\delta)\mu(b)$.
\end{claim}

\begin{proof}
Let $x\in\{0,1\}^{F_{r}}$ and let $t' = \per(x)\leq 2r+1$ be the minimal period of $x$. Assume first that for every $r \leq k$ it holds that $\mu(L^{t'} \cyl{x^k} \cap \cyl{x^k})< (1-\delta)\mu(\cyl{x^k})$. 
Applying this with $k=r, r+t',r+2t',\ldots$ we deduce that for every $i\geq 1$, 
    \[ \mu(\cyl{x^{r+it'}}) \leq \mu(L^{t'} \cyl{x^{r+(i-1)t'}} \cap \cyl{x^{r+(i-1)t'}}) < (1-\delta)\mu(\cyl{x^{r+(i-1)t'}})\;, \]
    where the first inequality holds by the inclusion $\cyl{x^{r+it'}}\subseteq L^{t'} \cyl{x^{r+(i-1)t'}} \cap \cyl{x^{r+(i-1)t'}} $ and the second is our assumption. Iterating, there is an $i = \Theta(\log(1/\upsilon)/\delta)$ such that $\mu(\cyl{x^{r+it'}})\leq \upsilon\mu(\cyl x)$. Therefore, $x$ satisfies the first case of the claim (with $\ell_x=r+it'$). 
    
    Otherwise, choosing $C_2$ as a large enough constant there must be some $r\leq \ell'_x \leq C_2 r \log(1/\upsilon)/\delta $ such that setting $b=\cyl{x^{\ell'_x}}$, $\mu(L^{t'} b \cap b)\geq  (1-\delta)\mu(b)$. 
    Since $t'$ is the minimal period of $x$, by item (i) in Lemma~\ref{lem:per_ext_properties} we have that $\cyl{x}$ is $F_{t'-1}$-independent, and thus so is $\cyl{x^{\ell'_x}}$.
    Consequently, $\Tow(b,t')$ is a tower as needed. 
\end{proof}

%

Let
\[E_0=T_0=\emptyset\qquad \text{and}\qquad S_0=\big\{x\in \{0,1\}^{F_{3t}}:\ 1\leq \per(x)\leq t \big\}\;.\]
Define inductively $E_i\supseteq E_{i-1}\supseteq \cdots\supseteq E_0$, $T_i\supseteq T_{i-1}\supseteq \cdots\supseteq T_0$ and $S_i\subseteq S_{i-1}\subseteq \cdots\subseteq S_0$, for $i\geq 1$, as follows. If $S_{i-1}$ is empty then the process stops and we set $E=E_i$, $T=T_i$. Otherwise, let $x\in S_{i-1}$. Apply Claim~\ref{claim:62}, with $r=3t$, to $x$. If $x$ falls in the first case then set $T_{i}=T_{i-1}$, $E_i = E_{i-1} \cup  \cyl{x^{\ell_x}}$ and  $S_i=S_{i-1}\backslash \{x\}$. If $x$ falls in the second case then set $T_i = T_{i-1}\cup \Tow(\cyl{x^{\ell_x}},\per(x))$, $E_i=E_{i-1}$ and $S_i=S_{i-1}\backslash \{x,(L\ext(x))^{3t},\ldots,(L^{(j-1)}\ext(x))^{3t}\}$. 

Firstly, it is clear that this process yields $E$ such that 
\[\mu(E)\leq \sum_{x\in \{0,1\}^{F_{3t}}} \upsilon \mu(\cyl{x})\leq \upsilon\;.\]
Furthermore, $E$ and the bases of towers in $T$ are both $\max_{x\in \{0,1\}^{F_{3t}}} \ell_x$-defined. 
We next verify that all towers in $T$ are disjoint. Let $x_1\in \{0,1\}^{F_{3t}}$ be an element of period $j_1=\per(x_1)$ added to $T$ at step $i_1$, and $x_2$ an element of period $j_2=\per(x_2)$ added to $T$ at step $i_2>i_1$. Suppose that $\Tow(\cyl{x_2},j_2)$ intersects $\Tow(\cyl{x_1},j_1)$. Then because $j_1,j_2\leq t$ there must be an $i\in F_t$ such that $\cyl{x_2}\cap L^i \cyl{x_1}\neq \emptyset$. By restricting to $F_{2t}$ we deduce that $x_2^{2t}=(L^i (x_1))^{2t}$. Since both $x_1$ and $
x_2$ have period at most $t$ by assumption, it follows from item (iii) in Lemma~\ref{lem:per_ext_properties} that $x_2=(L^i \ext(x_1) )^{3t}$, which contradicts how the sets $S_*$ are constructed. 

Finally, it remains to verify that this process indeed ends up covering $X_{per}^\ell(t)$. For every $y\in X_{per}^\ell(t)$, by Remark \ref{rem:min_per} it holds that $\per(y^\ell)\leq t$.
Assuming $\ell \geq 3t$, which is satisfied as long as $C_1$ is chosen large enough, letting $x=\pi_{3t}(y)$ it follows from  clause $(iii)$ in Lemma~\ref{lem:per_ext_properties} that $y^\ell=(\ext(x)^\ell$. Then $y \in \cyl{x^{\ell_x}}$ for every $\ell\geq \ell_x \geq 3t$, and hence $y$ is covered by the procedure. 

The asserted complexity bounds follow since, as already observed, $E$ and the bases of towers in $T$ are both $\max_{x\in \{0,1\}^{F_{3t}}} \ell_x$-defined, which is $O(t \log(1/\upsilon)/\delta)$ by Claim~\ref{claim:62}; and the towers have height at most $6t + 1$.
\end{proof}

\subsection{Covering the aperiodic sequences} \label{subsec:aper}

\subsubsection{A polynomial marker lemma}


In this section we derive the following marker lemma, which applies the general efficient marker machinary developed in Section \ref{sec:marker-efficient} to our setup.  

\begin{proposition}[Polynomial marker lemma] \label{prop:marker}
Let $\Xi \subset X_{\aper}^\ell(t) \subset X = \{0,1\}^{\mathbb Z}$ be a $K$-defined subset. 
Then there exists a clopen subset $Z \subset\Xi$ such that:
\begin{enumerate}
    \item $Z$ is $F_t$-independent.
    \item The collection $L^j Z, j \in F_t$ covers $\Xi$.
    \item  $Z$ is $R$-defined for  $R=O(\max(K,\ell) + t^3+t\log^*(\ell))$. 
\end{enumerate}
\end{proposition}

\begin{proof}
    Consider the continuous map $\varphi: \Xi \to [N]$ taken to be  $\varphi=\pi_\ell$, where $N = 2^{2\ell+1}$.
    By the definition of $X^\ell_{\aper}(t)$, we have that $\varphi$ is a $F_t$-proper coloring, which is moreover $F_{\max(K,\ell)}$-defined with respect to $\pi_0$.
    
    We can apply Proposition \ref{prop:marker_general} for $F = F_t$, there exists $Z \subset \Xi$ which satisfies $(1),(2)$, and is $B_F(r(N) + O(|F|^2))$-defined with respect to $\varphi$.
    Notice that since $\Gamma = \mathbb Z$, in this case we have:
    $$B_F(r(N) + O(|F|^2))  = F_{t\cdot(r(N)+O(|F|^2))} = F_{O(t^3+t\log^*(\ell))} .$$
    
    Since $\varphi$ is $F_{\max(K,\ell)}$-defined with respect to $\pi_0$ and $Z$ is $F_{O(t^3+t\log^*(\ell))}$-defined with respect to $\varphi$, it follows that $Z$ is $F_{O(\max(K,\ell) + t^3 + t\log^*(\ell))}$-defined with respect to $\pi_0$, certifying $(3)$.
\end{proof}

\subsubsection{Constructing Marker Sets}
\label{sec:marker}

Continuing with the notations of the previous section, we fix $t \in \N,\upsilon,\delta>0$ and let $\Xi_{\per}$ be the covering of $X_{per}^\ell (t)$ constructed in Lemma \ref{lem:covering_per_seq}.
Recall that $\Xi_{\per}$ depends on $\upsilon, \delta, t$, and is $O(t\ell)$-defined where $\ell = \Theta (t\log(1/\upsilon)/\delta)$.

Set $\Xi_{\aper} = \{0,1\}^\mathbb Z \setminus \Xi_{\per}$ to be the complement, and note that $\Xi_{\aper} \subset X_{aper}^\ell (t)$ is $O(t\ell)$-defined.

We now apply the polynomial marker lemma, Proposition~\ref{prop:marker}, to $\Xi_{\aper}$. Since $\Xi_{\aper}\subset X_{\aper}^\ell(t)$ and $\Xi_{\aper}$ is $O(t\ell)$-defined, the proposition applies with $K=O(t\ell)$ and yields a clopen marker set $Z\subset\Xi_{\aper}$ such that:
\begin{enumerate}
    \item $Z$ is $F_t$-independent.
    \item The collection $L^jZ$, $j\in F_t$, covers $\Xi_{\aper}$.
    \item $Z$ is $R_Z$-defined for $R_Z=O(t\ell+t^3+t\log^*(\ell+2))$.
\end{enumerate}
Choose a constant $C_M>0$ large enough such that letting 
\begin{equation}\label{eq:M-poly}
    M=C_MtR_Z\;,
\end{equation}
both $Z$ and $\Xi_{\aper}$ are $\lceil M/2t\rceil$-defined.

In Section~\ref{sec:mtot} below we will use $Z$ to construct towers that cover $\Xi_{\aper}$. However, before we do so we need to control how much $\cup_{i \in F_t}L^iZ$  intersects the previously constructed towers in  $\Xi_{\per}$. This is done in Lemma~\ref{lem:complement-bound}. In the proof, we will be repeatedly using the assumption that $\mu$ is $(M,\eta)$-invariant (Definition \ref{def:approx_inv_meas}), that is:
\begin{equation*} 
    \sum_{b \in \{ 0,1\}^{F_M}} |\mu(L\cyl{b}) - \mu(\cyl{b}) | \leq \eta\;.
\end{equation*}
First we collect a few consequences of approximate invariance that will be useful in the analysis to come. Recall that the function $\mu(A\triangle B)$ defines a metric on measurable sets, and $|\mu(A) - \mu(B)|\leq\mu(A\triangle B)$.
Further, we have the following useful inequality: $|\mu(A\cap C) - \mu(B\cap C)| \leq \mu(A \triangle B)$.

\begin{lemma}\label{lem:approx_inv_measures}
    Let $\mu$ be an $(M,\eta)$-invariant measure on $X = \{0,1\}^\Z$ for some $M,\eta >0$.
    Fix $k \in \N$ and $q \in \Z$ such that $|q|+k \leq M$, and let $\{b_i\}$ be a pairwise disjoint collection of  $k$-defined subsets and $b = \bigcup_i b_i$. 
    \begin{enumerate}
        \item  The following holds: $$\sum_i|\mu(L^qb_i) - \mu(b_i)| \leq |q|\eta\;.$$
        \item Assume that $\sum_i \mu(L^q b_i \cap b_i) \geq (1-\delta) \mu(b)$, then
        $$\sum_i\mu(L^q b_i \triangle b_i) \leq 2\delta \mu(b) + |q|\eta\;.$$
        Moreover, if $m\in \Z$ is such that $|qm|+k \leq M$ then 
        $$\sum_i\mu(L^{qm}b_i \triangle b_i) \leq   2(|m|\delta \mu(b) + |q|m^2 \eta)\;.$$ 
    \end{enumerate}
\end{lemma}

\begin{proof}
  Throughout the proof we will assume $q \geq 1$, for $q \leq 0$ the proof is similar.
    The proof of $(i)$ for $q=1$ follows from the definition of total variation distance. Since $q+k \leq M$, the general case follows by induction -- Indeed, by the triangle inequality,  $|\mu(L^qb_i) - \mu(b_i)|\leq | \mu(L^qb_i) - \mu(L^{q-1}b_i)| +| \mu(L^{q-1}b_i) - \mu(b_i)|$ for all $i$, as such:
    $$ \sum_i |\mu(L^qb_i) - \mu(b_i)|\leq \eta +\sum_i| \mu(L^{q-1}b_i) - \mu(b_i)|\;, $$
    where we use the base case for the disjoint collection $\{L^{q-1} b_i\}$.
    For $(ii)$, note that we have:
    \begin{align*}
        \sum_i\mu(L^q b_i \triangle b_i) &= \sum_i\left(\mu(L^q b_i) +\mu(b_i) - 2\mu(L^q b_i \cap b_i))\right) \\
        &\leq  \sum_i 2\left(\mu(b_i) - \mu(L^q b_i \cap b_i)\right) + q\eta \leq  2\delta \mu(b)+ q\eta\;,
    \end{align*}
    where we used $(i)$ in the second inequality, and the extra assumption in the last inequality.
    For the moreover assertion, assume that $m \geq 0$ and $qm+k \leq M$ (If $m < 0 $ the proof is similar). We apply the triangle inequality for the metric $\mu(A\triangle B)$, and obtain:
    $$\sum_i\mu(L^{qm}b_i  \triangle b_i ) \leq \sum_i\mu(L^{qm}b_i  \triangle L^{q(m-1)}b_i ) + \sum_i\mu(L^{q(m-1)}b_i  \triangle b_i )\;.$$
    Since $qm+k \leq M$, we have:
    \begin{align*}
    \sum_i\mu(L^{qm}b_i \triangle L^{q(m-1)}b_i) &= \sum_i\mu(L^{q(m-1)}(L^q b_i \triangle b_i)) \\
    &\leq \sum_i\mu(L^qb_i \triangle b_i) + 2q(m-1)\eta \leq 2\delta \mu(b) + 2qm\eta\;,
    \end{align*}
    where we used $(i)$ twice in the second line (separately for $\{ L^qb_i \setminus b_i\}$ and $\{ b_i \setminus L^q b_i\}$), and $(ii)$ in the last inequality.
    Now, arguing by induction, we get that $\sum_i\mu(L^{qm}b_i \triangle b_i) \leq 2m \delta  \mu(b) + 2qm^2\eta$, proving the desired bound.
\end{proof}


While $Z \subset \Xi_{\aper}$, there is no reason for $L^kZ$ to also be contained in $\Xi_{\aper}$ for $k \neq 0$.
Nevertheless, we can bound the measure of how much translates of $Z$ can escape $\Xi_{\aper}$.

\begin{lemma}\label{lem:complement-bound}
    Let $k_1,k_2\in F_{2t}$. Then $\mu(L^{k_1} Z \setminus L^{k_2} \Xi_{\aper}) = O(t^4\eta+ t^3\delta + \upsilon)$.  
\end{lemma}

\begin{proof}
    Since by assumption on $M$, $Z$ and $\Xi_{\aper}$ are both $\lceil M/2t\rceil$-defined, we have by part $(i)$ of Lemma~\ref{lem:approx_inv_measures} that 
    $$ \mu(L^{k_1} Z \setminus L^{k_2} \Xi_{\aper}) \leq \mu(L^{k}Z \setminus\Xi_{\aper}) + O(t\eta)\;,$$
    where $k = k_1 - k_2$.
    Note that $L^{k}Z \setminus\Xi_{\aper} \subset \Xi_{\per}$ is covered by the approximately closed towers constructed in Lemma \ref{lem:covering_per_seq} and the error set $E$, so it is enough to show the following: 
    
    \begin{equation}\label{eq:mutv-0}
      \sum_b \sum_{i=0}^{j-1}  \mu (L^k Z \cap L^ib) = O(t^4\eta+  t^3\delta)\quad \text{for every $k \in F_{4t}$}\;,
    \end{equation}
    where the outer summation over $b$ ranges over all bases $b$ for a tower $\Tow(b,j)$ in the covering by towers of $\Xi_{\per}$ from Lemma~\ref{lem:covering_per_seq}, where $j \leq 6t + 1$.
    This is because $\mu(E) < \upsilon$, so~\eqref{eq:mutv-0} implies the claim. 

To prove (\ref{eq:mutv-0}), we will apply Lemma \ref{lem:approx_inv_measures} several times. In all applications, the fact that the pairwise disjoint towers in $\Xi_{\per}$ are $O(t\ell)$-defined is used.
First, apply part $(i)$ of Lemma~\ref{lem:approx_inv_measures} to see that:
$$\sum_b \sum_{i=0}^{j-1}  \mu (L^k Z \cap L^ib) \leq \sum_b \sum_{i=0}^{j-1}  \mu ( Z \cap L^{i-k}b) + O(t\eta )\;.$$

    We will divide the sum into parts according to the height of the towers. For $j \leq 6t + 1$, let $\mathscr B_j$ be the set of bases for towers in $\Xi_{\per}$ which have height exactly $j$.  
     Fix $j \leq 6t+1$ and fix $0\leq i <j$. Note that $|i-k| \leq 10t + 1$ and let $q_i,s_i$ be such that $q_ij +s_i = (i-k)$ and $0\leq s_i\leq j-1$. Note that also $|q_ij| \leq11t + 1$.
    Since all the towers appearing are $\delta$-closed, we can apply $(ii)$ in Lemma \ref{lem:approx_inv_measures} and conclude that
    $$ \sum_{b \in \mathscr B_j}\mu(L^{q_ij} b \triangle b) \leq 2(|q_i| \delta  \sum_{b \in \mathscr B_j}\mu(b) + jq_i^2\eta) = O\Big(t^2 \Big(\delta \sum_{b \in \mathscr B_j}\mu(b)+ \eta\Big)\Big)\;.$$ 
    So, by part $(i)$ of Lemma~\ref{lem:approx_inv_measures}, we have:
    $$ \sum_{b \in \mathscr B_j} \mu(L^{i-k} b \triangle L^{s_i}b) = \sum_{b \in \mathscr B_j} \mu(L^{q_ij + s_i} b \triangle L^{s_i}b) = O\Big(t^2 \Big(\delta \sum_{b \in \mathscr B_j}\mu(b)+ \eta\Big)\Big)\;.$$
Summing over $i \leq j$, we obtain the following bound
    $$ \sum_{b \in \mathscr B_j} \sum_{i=1}^{j-1} \mu(L^{i-k} b \triangle L^{s_i}b)  = O\Big(t^3 \Big(\delta \sum_{b \in \mathscr B_j}\mu(b)+ \eta\Big)\Big)\;.$$
   Consequently, by the inequality $|\mu(C \cap A) - \mu(C\cap B)| \leq \mu(A\triangle B)$, we have 
    \begin{align*}
        \sum_{b\in \mathscr B_j} \sum_{i=0}^{j-1}  \mu ( Z \cap L^{i-k}b) &\leq \sum_{b \in \mathscr B_j} \sum_{i=0}^{j-1}  \mu ( Z \cap L^{s_i}b) + O\Big(t^3 \Big(\delta \sum_{b \in \mathscr B_j}\mu(b)+ \eta\Big)\Big)\\
        &=  O\Big(t^3 \Big(\delta \sum_{b \in \mathscr B_j}\mu(b)+ \eta\Big)\Big)\;.
    \end{align*}
The last inequality used the fact that $Z \subset \Xi_{\aper}$, so it does not intersect any of the towers appearing in $\Xi_{\per}$. 
Finally, if we sum over $j \leq 6t+1$, we obtain:
    \begin{align*}
        \sum_{b} \sum_{i=0}^{j-1}  \mu ( Z \cap L^{i-k}b) \leq  O\Big(t^3 \delta \sum_{j=0}^{6t+1}\sum_{b \in \mathscr B_j}\mu(b)+ t^4\eta\Big) = O(t^3\delta + t^4\eta)\;,
    \end{align*}
    where we use the fact that the collections $\mathscr{B}_j$ consist of pairwise disjoint sets. 
\end{proof}

\subsubsection{From Markers to Towers}
\label{sec:mtot}

Let us now take $W = L^{-t} \left(\bigcap_{i\in F_t} \left(Z \cap L^{-i}(\Xi_{\aper})\right)\right) \subset X$.
Note that $W$ is also $O(M)$-defined, since $Z$ is, and further, $\bigcup_{k=0}^{2t}L^kW \subset \Xi_{\aper}$ by definition.
We start by showing that $\mu(\Xi_{\aper} \setminus \bigcup_{k=0}^{2t}L^kW) = \mu(\Xi_{\aper} \setminus \bigcup_{k \in F_t}L^k\left(\bigcap_{i\in F_t} \left(Z \cap L^{-i}(\Xi_{\aper})\right)  \right))$ is small.
Indeed, since $\Xi_{\aper} \subset \bigcup_{k \in F_t}L^kZ$ by the saturation property of the marker set (item (ii) in Lemma~\ref{prop:marker}), it will be enough to show that for all $k \in F_t$, the following measure is small:
$$\mu\Big( L^k\Big(Z \setminus \bigcap_{i\in F_t} \big(Z \cap L^{-i}(\Xi_{\aper})\big )\Big)\Big) = \mu\Big(\bigcup_{i \in F_t} L^kZ \setminus L^{k-i}(\Xi_{\aper}) \Big).$$
However, by Lemma \ref{lem:complement-bound} and the union bound, the last expression is $O(t^5(\eta+ \delta + \upsilon))$.
Consequently, we have by union bound again:
$$\mu\Big(\Xi_{\aper} \setminus \bigcup_{k=0}^{2t}L^kW\Big) \leq  \mu\Big(\bigcup_{k \in F_t} L^k\Big( Z \setminus \bigcap_{i\in F_t} \left(Z \cap L^{-i}(\Xi_{\aper})\right)\Big)\Big) =  O(t^6(\eta+ \delta+\upsilon))\;.$$
Set $\Upsilon = E \cup \left( \Xi_{\aper} \setminus \bigcup_{k=0}^{2t}L^kW \right)$ to be our final error set, by the preceding bounds, we have the final error bound $\mu(\Upsilon) = O(t^6(\eta+ \delta + \upsilon))$.

It remains to cover $\bigcup_{k=0}^{2t}L^kW$ by disjoint towers. We will do so by constructing Kakutani-Rokhlin towers as follows:
Consider the function $g: W \to \mathbb N$ defined as follows: let $g(x)= \min \{ 1 \leq j \leq 2t | \; L^jx \in W \}$ if the minimum is taken over a nonempty set, otherwise let $g(x) = 2t+1$.  Consider the clopen sets $W_j = g^{-1}(j)$ for $1 \leq j \leq 2t+1$ which form a partition of $W$.
Firstly, notice that we have the disjoint tower decomposition:
$$ \bigsqcup_{j=1}^{2t+1}\bigsqcup_{i=0}^{j-1}L^iW_j = \bigcup_{k=0}^{2t}L^kW\;,$$
so that $W_j$ is the base of a tower of height $j$. This follows directly from the definition of $g$ (the fact that if $b\subset W$ is such that $L^i b \cap L^j b\neq \emptyset$ for some $0<i<j$ then also $b\cap L^{j-i}b \neq \emptyset$ is used to justify that the $L^i W_j$ are indeed disjoint for $0\leq i < j$). 
Further, note that $g$ is continuous, and that in fact each such tower is $O(M)$-defined.
Secondly, we claim that $W_j = \emptyset$ for each $j \leq t$. Indeed, this follows from the fact that $Z$ is $F_t$-independent (by item (i) from Lemma~\ref{prop:marker}), and so is $W = L^{-t}(Z)$.
Consequently, all towers constructed in this way have height at least $t$.

\subsection{Final Tower Decomposition}
We have covered $X$ as a disjoint union
$$ X = \Upsilon \sqcup (\Xi_{\per} \setminus E) \sqcup \bigsqcup_{j=t+1}^{2t+1}\bigsqcup_{i=0}^{j-1}L^iW_j\;,$$
where $\Upsilon$ is an error set with $\mu(\Upsilon) = O(t^6(\upsilon + \delta+ \eta))$, $\Xi_{\per} \setminus E$ is covered by disjoint $M$-defined $\delta$-closed towers of height $\leq 6t + 1$, and $W_j$ are bases for $M$-defined towers of height $j$ with $t+1 \leq j \leq 2t+1$, where $M=C_MtR_Z$ is as in \eqref{eq:M-poly}. According to Lemma~\ref{lem:covering_per_seq} we can take $\ell = \lceil C_1 t\log(1/\upsilon)/\delta\rceil$, and hence
\[
    M=O\left(t^3\frac{\log(1/\upsilon)}{\delta}+t^4+t^2\log^*\left(\frac{t\log(1/\upsilon)}{\delta}+2\right)\right).
\]
This gives a tower decomposition satisfying $(i)$ and $(ii)$ as stated in Proposition \ref{prop:decomp}.
A last modification is needed to guarantee $(iii)$ in Proposition \ref{prop:decomp}:
Note that the towers $\Tow(b,j)$ falling in the first case of $(ii)$, that is, the towers that are of height less than $t$ and are $\delta$-closed, all appear in the decomposition of $\Xi_{\per} \setminus E$ constructed by Lemma \ref{lem:covering_per_seq}. 
The lemma guarantees that in this case,  $\pi_t(b)$ is a singleton, so that part $(iii)$ is satisfied and no further modification is needed.
We are left with the towers $\Tow(b,j)$ that have height $t \leq j \leq 6t+1$. In this case, we can split the base $b$ to a disjoint union $b = \bigcup_{x \in \{0,1\}^{F_j}} b_x$ by taking $b_x = b \cap \cyl{x}$ for $x \in \{0,1\}^{F_j}$. 
Each $b_x$ which is non-empty will be the base of a $M$-defined tower $\Tow(b_x, j)$, which together cover $\Tow(b,j)$, and further $\pi_j(b_x) = \{x\}$ is a singleton.
Replacing $\Tow(b,j)$ by this collection will result in a tower decomposition which satisfies all the statements in Proposition \ref{prop:decomp}.

\section{Application to Stability rate and Stability radius}
\label{sec:radius}


In this section we recall some general properties of the stability rate and stability radius growth functions of \cite{becker2021abelian, dogon2024characters} mentioned in the introduction (Definitions \ref{def:local_global_defect},  \ref{def:SRad}), and then apply Theorem \ref{thm:main} to deduce upper bounds in the case of the lamplighter group.
Let $\Gamma$ be a group with a finite generating set $S$. For each $r\in \N$ denote by $B_S(r)$ the ball of radius $r$ around the identity element of $\Gamma$ with respect to the word metric.

\begin{definition}
[Stability Rate] \label{def:local_global_defect}
Fix $\Gamma = \langle S \rangle$ a finitely generated group, and $r \in \mathbb{N}$.
The $r$-\emph{\textbf{local defect}} of a map $\varphi:B_S(r) \to U(d)$ is
$$L_{S,r}(\varphi) = \max_{g,h, g\cdot h\in B_S(r)} \; \|\varphi(gh) -\varphi(g)\varphi(h)\|_{\HS}.$$
The \textbf{\emph{global defect}} of a map $\varphi: S \to U(d)$ is
$$ G_S(\varphi) = \inf\{ \max_{g \in S} \|\varphi(g) -\pi(g)\|_{\HS} | \;\; \pi: \Gamma \to U(d) \text{ is a unitary representation} \}.$$
The \textbf{\emph{$r$-stability rate}} is the function $\SRate^{S,r}_\Gamma:(0,\infty) \to (0,2]$ defined as follows: $\SRate^{S,r}_\Gamma(\kappa)$ is the supremal $\delta > 0$ such that for every $\varphi: B_S(r) \to U(d)$ with $L_{S,r}(\varphi) < \delta$, we have $G_S(\varphi) < \kappa$.
If no such $\delta>0$ exists, set $\SRate_{\Gamma}^{S,r}(\kappa)=0$.


\end{definition}

In essence, stability asks for maps with small local defect to also have small global defect.
If $\Gamma$ is finitely presented, one may fix $r \in \N$ large enough and the function $\SRate_\Gamma^{S,r}$ measures stability, recovering\footnote{More precisely, the definition of stability rate in our setting recovers the \emph{generalized inverse} of the stability rate of Becker-Moshieff.} the approach of \cite{becker2021abelian} (Lemma \ref{lem:SR_fin_presented} below).
However, if $\Gamma$ is infinitely presented, the radius parameter $r$ needs to be quantified as well, leading to the following notion.

\begin{definition}[Stability radius growth \cite{dogon2024characters}]
    \label{def:SRad}
    The \textbf{\emph{Hilbert--Schmidt stability radius growth}} of a group $\Gamma$ with respect to a finite generating set $S$ is the function $\SRad^S_\Gamma: (0, \infty) \to \mathbb N \cup \{ \infty \}$ defined as follows, for $t>0$: 
    $$ \SRad_\Gamma^S(t) = \min\{r\in \N | \; \SRate_\Gamma^{S,r}(1/t) > 0  \}. $$
    If no such $r \in \mathbb N$ exists, set $\SRad_\Gamma^S(t) = \infty$.

   %
\end{definition}

 The stability radius growth is a non-decreasing function, and it turns out that the group $\Gamma$ is \textbf{\emph{Hilbert--Schmidt stable}} (in the sense of \cite{HS_grp, BLT}) if and only if $\SRad_\Gamma^S(t)$ is finite for all $t > 0$.
 Further, if $\Gamma = \langle S|E\rangle$ is a finitely presented Hilbert--Schmidt stable group, then $\SRad_\Gamma^S$ is bounded from above by the maximal length of a relation appearing in $E$.


Up to asymptotic equivalence, $\SRad_\Gamma^S$ is independent of the choice of the particular  generating set \cite[Lemma 12.3]{dogon2024characters}.
This justifies denoting its asymptotic equivalence class by $\SRad_\Gamma$.
Our main result implies a polynomial upper bound on the stability radius, as well as the stability rate, in the case of the \emph{lamplighter group}:

Recall the notion of equivalence between monotone functions which was part of our standing notation on p. \pageref{notations}.
We begin by clarifying the situation for finitely presented groups.

\begin{lemma} \label{lem:SR_fin_presented}
If the group $\Gamma$ is finitely presented and Hilbert--Schmidt stable then the function $\SRad_\Gamma^S$ is bounded.
Fixing $r$ to be such a bound, $\SRate_\Gamma^{S,r}$ is the generalized inverse of the stability rate introduced in \cite[Definition 1.4]{becker2021abelian}, up to asymptotic equivalence which may depend on $r$.
\end{lemma}

The proof of the lemma follows from \cite[Lemma 12.1]{dogon2024characters} and \cite[Definition 1.4]{becker2021abelian}.
Importantly, up to asymptotic equivalence, the function $\SRad_\Gamma ^S$ turns out to be independent of the generating set.

\begin{lemma}[\cite{dogon2024characters}, Lemma 12.3]\label{lem:SR_independent_of_gen}
If $S_1$ and $S_2$ are a pair of finite generating sets for the group $\Gamma$ then $\SRad_\Gamma^{S_1} \approx \SRad_\Gamma^{S_2}$.
\end{lemma}

We now clarify the relation between Theorem \ref{thm:main} and quantitative stability of the lamplighter group.
For this we use the presentation from Section \ref{sec:rep-theory}
\[ \Gamma = \big\langle a,t \mid a^2,[a,t^n a t^{-n}], n\in\mathbb{Z}\big\rangle\;,\]
with the generating set $S=\{a,t\}$.

\begin{lemma}\label{lem:gamma-rep}
Let $r\geq 2$ and $\varphi: B_S(r) \to U(d)$ be a map with $L_{S,r}(\varphi) < \delta$.
Then there exists a unitary $T\in U(d) $ and a Hermitian involution $A\in U(d)$ such that $\varphi(t)=T$ and $\|\varphi(a)-A\|_{\HS}\leq \delta$. Letting $A_i = T^{-i} A T^{i}$ for $i\in \mathbb{Z}$, we have $\|[A_i,A_j]\|_{\HS} =O(r\delta)$ for all $i,j\in \mathbb{Z}$ with $|i|, |j| \leq (r-2)/4 $.
\end{lemma}

\begin{proof}
Since $r\geq 2$, $\|\varphi(a)^2-\id\|_{\HS}\leq \delta$, which by a simple rounding argument (see e.g.~\cite[Proposition 1.4]{DGLT}) gives that there is a Hermitian $A$, $A^2=\Id$, such that 
\begin{equation}\label{eq:gamma-rep-0}
\|\varphi(a)-A\|_{\HS}\leq \delta\;.
\end{equation}
Furthermore, for any $i\in\mathbb{Z}$ such that $|i|\leq r$,
\begin{align}
    \|\varphi(t^i)-T^i\|_{\HS} &= \|\varphi(t^i)-\varphi(t)^i\|_{\HS} \leq |i|\delta\; = O(r\delta),\label{eq:gamma-rep-1}
\end{align}
by the triangle inequality and successive application of $L_{S,r}(\varphi) < \delta$.

For any $i,j\in\mathbb{Z}$, provided that $2|i-j|+2\leq r$:
\begin{align*}
\|[A_i,A_j]\|_{\HS} &= \| [T^{j-i}A T^{i-j} , A] \|_{\HS} \\
&\leq \| [\varphi(t^{j-i})\varphi(a) \varphi(t^{i-j}) ,\varphi(a) ] \|_{\HS} + O(r\delta) \\
&= O(r\delta) \;,    
\end{align*}
where the second line is by the triangle inequality and~\eqref{eq:gamma-rep-0} and~\eqref{eq:gamma-rep-1} and the last line applies $L_{S,r}(\varphi) < \delta$ twice.   
\end{proof}

We can now prove Theorem~\ref{thm:rate-radius-lamplighter} announced in the introduction. 

\begin{proof}[Proof of Theorem~\ref{thm:rate-radius-lamplighter}]
Let $0<\kappa\leq 1/2$ and let $M_\kappa=\lceil C\kappa^{-20}\log(2/\kappa)\rceil$ be the radius parameter in Theorem~\ref{thm:main}. Put $r_\kappa=8M_\kappa+4$. If $\varphi:B_S(r_\kappa)\to U(d)$ has local defect at most $c\kappa^7/M_\kappa^3$, then Lemma~\ref{lem:gamma-rep}, after adjusting constants and using the choice of $r_\kappa$ to control commutators up to $2M_\kappa$, produces unitaries $A,T$ satisfying the hypotheses of Theorem~\ref{thm:main} with target error $\kappa/3$ and with $M=M_\kappa$.
Theorem~\ref{thm:main} therefore gives a genuine lamplighter representation whose values on $a,t$ are within $O(\kappa)$ of $\varphi(a),\varphi(t)$. Hence
\[
    \SRate_{\Gamma}^{S,r_\kappa}(\kappa)\geq c\,\frac{\kappa^{67}}{\log(2/\kappa)^3},
\]
and $\SRad_\Gamma^S(1/\kappa)\leq r_\kappa=O(\kappa^{-20}\log(2/\kappa))$. Since $\log(2/\kappa)\leq C_\alpha\kappa^{-\alpha}$ for every fixed $\alpha>0$, this implies the coarser asymptotic bound $\SRad_\Gamma(r)\preceq r^{21}$.
\end{proof}


\begin{remark}
    Both the definitions of stability rate and stability radius growth make sense for stability with respect to arbitrary families of groups equipped with bi-invariant metrics (see for example \cite{BACHNER2026235}), and versions for \emph{flexible stability} \cite{BL,Dog} can be defined as well.
    The case of \emph{permutation stability} has been studied in \cite{becker2021abelian}, \cite{bradford2026groups}.
    Here finite dimensional unitary groups with the Hilbert--Schmidt metric are replaced with symmetric groups with the \emph{normalized Hamming distance}.
    Our methods do not seem to be naively adaptable to achieve an effective bound on the permutation stability of the lamplighter group, which was established qualitatively by \cite{LL1}. 
\end{remark}

\begin{remark}\label{remark:pathology}
   Other possible measures of stability exist, such as the one in \cite{bradford2026groups}, which combines both the stability rate and radius into one function, called the \emph{stability growth function}.
Theorem \ref{thm:main} can be used to give suitable effective bounds on these variants as well.
    Our reason for keeping the stability radius and stability rate separate is to avoid the following pathology:
    By allowing to measure the stability rate $\SRate_\Gamma^{S,r}(\kappa)$ on balls larger than the minimal one for which it is non-zero, one could potentially \emph{improve} the stability rate.
    This would prevent us from recovering Becker and Moshieff's original definition of stability rate in the finitely presented regime, although at present we do not have such an example.
\end{remark}

\appendix

\section{A linear lower bound}
\label{sec:linear-lower-bound}

We record a simple example showing that the stability radius of the lamplighter is at least linear. We remark that it is also possible to deduce this from applying \cite[Proposition 12.6]{dogon2024characters}, together with a gap between the LEF-growth and MAP-growth (see \cite[Definition 12.5]{dogon2024characters}) of the lamplighter.

\begin{proposition}
\label{prop:linear-lower-bound}
Let \(\Gamma=\Z/2\wr\Z\) and \(S=\{a,t\}\) as in Section \ref{sec:radius}, then
\[
    \SRad^S_\Gamma(t)\succeq t .
\]
\end{proposition}

\begin{proof}
Let
\[
    X=\begin{pmatrix}0&1\\ 1&0\end{pmatrix},
    \qquad
    Z=\begin{pmatrix}1&0\\ 0&-1\end{pmatrix},
\]
be the Pauli matrices, so that $[X,Z]=2XZ$. Set \(N=4m\), let \(\cH=\C^N\ot\C^2\), and let \(U = C_{N,1}\) be the cyclic shift matrix as in Definition \ref{def:rep-of-lamp}. Define \(T=U\ot 1\) and 
\(A=\bigoplus_{r\in \Z/N} B_r\), where
\[
B_r=
\begin{cases}
X, & 0\leq r<m,\\
1, & m\leq r<2m,\\
Z, & 2m\leq r<3m,\\
1, & 3m\leq r<4m.
\end{cases}
\]
Then $A,T \in U(\cH)$ and \(A^2=1\). Since \(T^{i}AT^{-i} = \bigoplus_{r\in \Z/N} B_{r+i\; (\textrm{mod $N$})} \) we have:
\[
    [A,T^{i}AT^{-i}]
    =
    \bigoplus_{r\in\Z/N} [B_r,B_{r+i}].
\]
It follows that
\[
    [A,T^{i}AT^{-i}]=0,
    \qquad 0\leq i\leq m.
\]
Set $\varphi(a) = A, \varphi(t) = T$. 
It follows from the semidirect product structure $\Gamma = \bigoplus_{\Z} \Z /2 \rtimes \Z$ that $\varphi$ extends (uniquely) to a \emph{partial homomorphism} $\varphi:B_S(R) \to U(\cH)$, for $R = \lfloor m/4 \rfloor$. 
That is, $\varphi(g)\varphi(h) = \varphi(gh)$ for all $g,h\in B_S(R)$, so that $L_{S,R}(\varphi) = 0$.
On the other hand, since \(\norm{[X,Z]}_{\HS}^2=2\), we get
\[
    \norm{[A,T^{2m}AT^{-2m}]}_{\HS}^2
    =
    \frac{1}{4m}\cdot 2m \cdot \norm{[X,Z]}_{\HS}^2
    =
    2.
\]
Thus
\begin{equation}
\label{eq:missing-relation-linear-lower}
    \norm{[A,T^{2m}AT^{-2m}]}_{\HS}=\sqrt{2}.
\end{equation}

Let $\kappa > 0$, and suppose that \(\rho:\Gamma\to U(\cH)\) is a true representation such that 
$$\norm{\varphi(a)-\rho(a)}_{\HS}\leq \kappa, \norm{\varphi(t)-\rho(t)}_{\HS}\leq \kappa.$$
For every \(n\geq 1\) we have similarly to Lemma \ref{lem:gamma-rep} that
\[
    \norm{\varphi(t)^{n}\varphi(a)\varphi(t)^{-n}-\rho(t)^{n}\rho(a)\rho(t)^{-n}}_{\HS}
    \leq (2n+1)\kappa.
\]
Since $\rho$ is a representation, by using the elementary estimate
\[
    \norm{[A,B]-[A',B']}_{\HS}
    \leq 2\norm{A-A'}_{\HS}+2\norm{B-B'}_{\HS}
\]
for contractions, we obtain
\[
     \norm{[A,T^{n}AT^{-n}]}_{\HS} = \norm{[\varphi(a),\varphi(t)^{n}\varphi(a)\varphi(t)^{-n}]}_{\HS}   
    \leq (4n+4)\kappa.
\]
Taking \(n=2m\) and using \eqref{eq:missing-relation-linear-lower} gives
\[
    \kappa\geq \frac{\sqrt{2}}{8m+4}.
\]
It follows that $G_S(\varphi) \geq \frac{\sqrt{2}}{8m+4}$, and so $\SRate_{\Gamma}^{S,r}\left(\frac{\sqrt{2}}{8m+4}\right) = 0$ for every $r \leq R$.
It follows that $\SRad_{\Gamma}^S(m) \succeq m$ as desired.

\end{proof}

\footnotesize
\bibliography{main}

\vspace{0.5cm}

\noindent{\textsc{Weizmann Institute of Science, Israel}}

\noindent{\textit{Email address:} \texttt{alon.dogon@mail.huji.ac.il}} \\

\noindent{\textsc{École Polytechnique Fédérale de Lausanne, Switzerland and Weizmann Institute of Science, Israel}}

\noindent{\textit{Email address:} \texttt{thomas.vidick@epfl.ch}} \\

\end{document}